\documentclass[11pt]{article}

\usepackage{amsmath, amssymb, amsthm, mathtools}
\usepackage[margin=1in]{geometry}
\usepackage{enumitem}
\usepackage{hyperref}
\usepackage{cleveref}
\usepackage{xcolor}
\hypersetup{
	colorlinks=true,
	linkcolor=blue!50!black, % Deep blue for internal links (sections, equations)
	citecolor=green!50!black, % Deep green for citations
	urlcolor=red!50!black % Deep red/maroon for web URLs
}
\usepackage{algorithm}
\usepackage{algorithmicx}
\usepackage{algpseudocode}

\newtheorem{theorem}{Theorem}[section]
\newtheorem{lemma}[theorem]{Lemma}
\newtheorem{corollary}[theorem]{Corollary}

\theoremstyle{remark}
\newtheorem{remark}[theorem]{Remark}

\newcommand{\Sph}{\mathbb{S}}
\newcommand{\R}{\mathbb{R}}
\newcommand{\Z}{\mathbb{Z}}

\newcommand{\eps}{\epsilon}
\newcommand{\cP}{\mathcal{P}}
\newcommand{\cA}{\mathcal{A}}
\newcommand{\cB}{\mathcal{B}}
\newcommand{\cD}{\mathcal{D}}
\newcommand{\cE}{\mathcal{E}}
\newcommand{\cF}{\mathcal{F}}

\newcommand{\cM}{\mathcal{M}}
\newcommand{\cN}{\mathcal{N}}
\newcommand{\cR}{\mathcal{R}}
\newcommand{\ip}[2]{\langle #1,#2\rangle}
\newcommand{\norm}[1]{\left\|#1\right\|}
\newcommand{\embeds}[1]{\overset{#1}{\hookrightarrow}}
\DeclareMathOperator{\supp}{supp}
\DeclareMathOperator{\poly}{poly}
\DeclareMathOperator{\polylog}{polylog}
\DeclareMathOperator{\range}{range}
\DeclareMathOperator{\diam}{diam}
\DeclareMathOperator*{\E}{\mathbb{E}}
\DeclareMathOperator*{\Var}{Var}
\DeclareMathOperator{\sgn}{sgn}

\title{Optimal Embeddings of Constant-Dimensional Subspaces of $L^p$ into $\ell_p^N$}
\author{Yi Li \\ Nanyang Technological University \\ \texttt{yili@ntu.edu.sg}}
\date{}
\begin{document}
\maketitle

\begin{abstract}
	For $d \geq 2$, $p \geq 1$ and $\epsilon > 0$, let $N_p(d,\epsilon)$ be the smallest integer $N$ such that every $d$-dimensional subspace of $L^p[0,1]$ admits a linear embedding
	into $\ell_p^N$ with distortion at most $1 + \epsilon$. For fixed $d\geq 2$
	and $p\geq 1$, the bound
	\[
		N_p(d,\epsilon) \lesssim_{d,p} \epsilon^{-2(d-1)/(d+2p)}
	\]
	is established. For $p \notin 2\mathbb{Z}$, this matches the known lower bound up to constant factors. For odd integers $p$, previous upper bounds with this exponent incurred additional logarithmic factors, except in the logarithm-free case $p = 1$; for non-integral $p$, no upper bound with this exponent was previously known. For even integers $p$, isometric embeddings of dimension independent of $\epsilon$ are known.
	
	For $p \notin 2\mathbb{Z}$, the proof approximates $|t|^p$ by a polynomial with a remainder of small total variation. The polynomial part contributes no error, while the error from the remainder is controlled by an integrated equatorial-band discrepancy estimate.
\end{abstract}

%!TeX root = main.tex
\section{Introduction}

Embedding finite-dimensional subspaces of $L^p$ into $\ell_p^N$ is a classical problem in the local theory of Banach spaces. Let $E \subseteq L^p[0,1]$ be a subspace. Given
$\eps > 0$, we say that $E$ admits a $(1+\eps)$-embedding into $\ell_p^N$ and
write $E\embeds{1+\eps}\ell_p^N$ if there exists a linear map $T: E \to \ell_p^N$ such that $\norm{x}_p \leq \norm{Tx}_p \leq (1+\eps)\norm{x}_p$ for all $x \in E$.
Define
\[
N_p(E,\eps) = \min\{ N: E\embeds{1+\eps}\ell_p^N \}
\]
and, given a dimension parameter $d$,
\[
N_p(d,\eps) = \sup \{N_p(E,\eps): E \text{ is a $d$-dimensional subspace
	of }L^p[0,1]\}.
\]
The embedding problem asks for the asymptotic behaviour of
$N_p(d,\eps)$ as a function of $d$ and $\eps$.

A comprehensive survey appears in~\cite{handbook:19}. Classical results give
\begin{equation}
	\label{eqn:N_p(d,eps)}N_p(d,\eps) \lesssim \eps^{-2}d^{\max\{1, p/2\}}\log^c
	(d/\eps),
\end{equation}
where $c > 0$ and the implicit constant are absolute; see, e.g., the monograph by Ledoux and Talagrand~\cite{LT91}.
Schechtman~\cite{schechtman:tight} proved for even integers $p$ that
$N_p(d,\eps) \lesssim p^{-2}\eps^{-2}\binom{d+p/2-1}{p/2}$ for $\eps\leq 1/p$, without additional logarithmic
factors. Very recently, Reis and Rothvoss~\cite{RR26} removed the logarithmic
factor for $p = 1$ altogether, obtaining $N_1(d,\eps) \lesssim d/\eps^2$.

It is folklore that the $d^{\max\{1, p/2\}}$ factor in \eqref{eqn:N_p(d,eps)}
is necessary, which can be seen, for instance, from the tightness of Dvoretzky's
theorem for finite-dimensional $\ell_p$ spaces. Regarding the dependence on
$\eps$, we remark that, when $p$ is even, it is possible to attain an isometric
embedding into $\ell_p^N$ with $N = \binom{d+p-1}{p} - 1$~\cite{handbook:21},
so there can be no dependence on $\eps$ for even $p$. For $p \notin 2\Z$, it
is known that $N_p(d,\eps) \gtrsim 1/(\eps^2 \polylog(1/\eps))$ when
$d \gtrsim \log(1/\eps)$~\cite{LWW21}, which indicates that the $1/\eps^2$ dependence in \eqref{eqn:N_p(d,eps)}
is tight up to logarithmic factors in this regime of $d$.

The situation changes dramatically when the dimension $d$ is a fixed constant.
In this regime the dependence on $\eps$ can be substantially better than $\eps^{-2}$
for $p = 1$. The optimal order is now known:
\[
	N_1(d,\eps) \asymp_d
	\left(\frac{1}{\eps}\right)^{\frac{2(d-1)}{d+2}}.
\]
The lower bound follows from the corresponding bound for embedding the Euclidean subspace
$\ell_2^d$, proved by Bourgain et al.~\cite{BLM89}.
The matching upper bound for $d = 2$ is classical. 
Matou\v{s}ek~\cite{mat96} proved the matching upper bound for $d\geq 5$
and obtained it up to logarithmic factors for $d = 3, 4$.
Siegel~\cite{Siegel25} recently removed these remaining logarithmic
factors.

For $p \notin 2\mathbb{Z}$, the proof technique in \cite{BLM89} can be
generalized to $p > 1$, yielding, as shown in \cite{LLW23},
\begin{equation}
	\label{eqn:LB}N_p(\ell_2^d, \eps) \gtrsim_{d,p} \left(\frac{1}{\eps}\right
	)^{\frac{2(d-1)}{d+2p}}.
\end{equation}
A near-matching upper bound for $N_p(d,\eps)$ was shown in the same paper \cite{LLW23}
for integers $p$,
\begin{equation}
	\label{eqn:N_p(d,eps)_const_d}N_p(d, \eps) \lesssim_{d,p}\left(\frac{1}{\eps}
	\right)^{\frac{2(d-1)}{d+2p}}\left(\log\frac{1}{\eps}\right)^{A_{d,p}},
\end{equation}
where $A_{d,p} > 0$ is a constant depending only on $d$ and $p$.

In this paper, we extend the constant-dimensional upper bound \eqref{eqn:N_p(d,eps)_const_d}
from integer $p$ to all $p \geq 1$. For every fixed $d$ and every $p \notin 2\mathbb{Z}$,
this settles the asymptotic dependence on $\eps$ up to constant factors. For
positive even integers $p$, as mentioned above, isometric embeddings exist without dependence on $\eps$.

\subsection{Main Results}

\begin{theorem}[Subspace embedding]	\label{thm:embedding}
	Let $d \ge 2$ and $p \ge 1$ be fixed. There exists
	a constant $C_{d,p} > 0$ such that, for every $\eps \in (0,1/2)$,
	\[
		N_p(d,\eps) \le C_{d,p}
		\eps^{-\frac{2(d-1)}{d+2p}}.
	\]
\end{theorem}

By the standard reduction, it suffices to prove a sparsification result with additive error
for finitely supported probability measures on the sphere; see, e.g.,~\cite{LLW23}.

\begin{theorem}[Sparsification with additive error]
	\label{thm:additive} Let $d \ge 2$ and $p \geq 1$ be fixed. Let
	\begin{equation}
		\label{eqn:mu}\mu = \sum_{i = 1}^M \rho_i\delta_{z_i}, \qquad \rho_i > 0,
		\quad \sum_{i = 1}^M\rho_i = 1,\quad z_i \in \Sph^{d-1},
	\end{equation}
	be a finitely supported probability measure on $\Sph^{d-1}$. Then for every $\eps
		\in (0,1/2)$, there exists a positive atomic probability measure
	\[
		\nu = \sum_{i = 1}^N w_i\delta_{y_i}, \qquad y_i \in \supp(\mu),\quad w_i
		> 0,
	\]
	with
	\[
		N \le C_{d,p}
		\eps^{-\frac{2(d-1)}{d+2p}},
	\]
	such that
	\begin{equation}
		\label{eqn:additive}\sup_{u \in \Sph^{d-1}}\left| \int_{\Sph^{d-1}}|\ip{u}{y}
		|^p\,d\nu(y) - \int_{\Sph^{d-1}}|\ip{u}{y}|^p\,d\mu(y) \right| \le \eps
		.
	\end{equation}
	Here, $C_{d,p}>0$ depends only on $d$ and $p$.
\end{theorem}

\begin{remark}
	If the input measure $\mu$ in Theorem~\ref{thm:additive} is given explicitly by the list ${(\rho_i,z_i)}_{i = 1}^M$, then for constant $d$ and $p$, the sparsified measure $\nu$ can be found, in the standard real-arithmetic model, by a randomized algorithm running in $\poly(M,1/\eps)$ time, with success probability at least $2/3$. See the details in Section~\ref{sec:algorithm}.
\end{remark}

Given Theorem~\ref{thm:additive}, we can prove Theorem~\ref{thm:embedding} as
follows.

\begin{proof}[Proof of Theorem~\ref{thm:embedding}]
	Since every $d$-dimensional subspace of $L^p[0,1]$ admits an arbitrarily close
	embedding into $\ell_p^M$ for a sufficiently large $M$, it suffices to
	consider a subspace parametrized by $A:\mathbb{R}^d \to \ell_p^M$. After
	reparametrizing $A$, John's theorem gives
	\[
		1 \le \|Au\|_p^p \le d^{p/2},\qquad \forall u \in \Sph^{d-1}.
	\]
	Write the nonzero rows of $A$ as $a_i = r_i z_i$, where $z_i \in \Sph^{d-1}$,
	and set
	\[
		W = \sum_i r_i^p, \qquad \mu = \frac{1}{W}\sum_i r_i^p\delta_{z_i}.
	\]
	Averaging $\|Au\|_p^p$ over $u \in \Sph^{d-1}$ shows that $W \asymp_{d,p}1$.

	Apply Theorem~\ref{thm:additive} to $\mu$ with accuracy $\delta = c_{d,p}\eps$,
	obtaining $\nu = \sum_{j = 1}^N w_j\delta_{y_j}$. Define
	\[
		Tu = \bigl((Ww_j)^{1/p}\langle u,y_j\rangle\bigr)_{j = 1}^N.
	\]
	Then, by choosing $c_{d,p}$ to be small enough,
	\[
		\sup_{u \in \Sph^{d-1}}\left|\|Tu\|_p^p - \|Au\|_p^p\right| \le W\delta
		\leq \eps \leq \eps \|Au\|_p^p.
	\]
	Define $T':\range(A) \to \ell_p^N$ by $T'(Au) = Tu$. This is well-defined since $A$ is injective. Then
	\[
		\left|\|T'v\|_p^p - \|v\|_p^p\right| \leq \eps \|v\|_p^p, \qquad  v \in \range(A).
	\]
	Rescaling
	$T'$ gives a $(1 + C_p\eps)$-embedding of $\range(A)$ into
	$\ell_p^N$. The bound on $N$ follows from Theorem~\ref{thm:additive}. Finally, applying the preceding argument with accuracy parameter $\eps/C_p$ proves the theorem.
\end{proof}

\subsection{Related work}
The classical upper bounds of $N_p(d,\eps)$ in~\eqref{eqn:N_p(d,eps)} are
established by iterative subsampling with the ``change of density'' technique~\cite{handbook:19}.
Suppose that one begins with a $d$-dimensional subspace $E \subseteq \ell_p^N$.
After placing $E$ in a suitable position by a change of density, one randomly subsamples
and appropriately rescales its coordinates. This produces a $d$-dimensional
subspace $F \subseteq \ell_p^{N'}$ with $N' \leq (3/4)N$ such that
$d(E,F) \approx 1 + \sqrt{d^{\max\{1,p/2\}}/N}$, where $d(\cdot,\cdot)$ denotes the
Banach--Mazur distance. The change of density ensures that the coordinates are
``spread out'', which is the regularity condition needed for the subsampling
argument. Thus, the ambient dimension is reduced by a constant factor in each
step. This reduction terminates when the distortion $1 + \sqrt{d^{\max\{1,p/2\}}/N}
		\approx 1 + \eps$, which corresponds to a target dimension of
$N \approx d^{\max\{1,p/2\}}/\eps^2$.

The change-of-density argument relies on Lewis' lemma~\cite{lewis}. Its
original proof is based on a determinant-maximization argument and does not
yield an efficient algorithm. A major algorithmic advance was made by Cohen and
Peng~\cite{CP15}, who characterized the Lewis weights as the fixed point of a
nonlinear map and gave an efficient iterative algorithm for approximating them
when $p < 4$. They also showed that, for $p \in [1,2]$, a single round of sampling
according to the Lewis weights yields an $\ell_p$-subspace embedding,
avoiding the iterative subsampling procedure. For $p > 2$, efficient high-precision
computation of Lewis weights was subsequently obtained by Fazel et al.~\cite{FLPS22}.
Earlier single-round sampling results for $p > 2$, due to Bourgain et al.~\cite{BLM89},
gave a target dimension of $\widetilde{O}(d^{p/2}/\eps^5)$, with a
substantially worse dependence on $\eps$ than \eqref{eqn:N_p(d,eps)}. More
recently, this was improved to the near-optimal target dimension $\widetilde O(
d^{p/2}/\eps^2)$ through single-round sampling constructions~\cite{ICLR24,Yasuda2024,Sun2024}.

When $d$ is a constant and $p = 1$, Bourgain and Lindenstrauss~\cite{BL88}
obtained the near-optimal bound
$N_1(\ell_2^d, \eps) = \widetilde{O}_d(\eps^{-\frac{2(d-1)}{d+2}})$
for embedding $\ell_2^d$ and the sub-optimal bound
$N_1(d,\eps) = \widetilde{O}_d(\eps^{-\frac{2(d-1)}{d+1}})$ for general
$d$-dimensional subspaces. Their arguments use concentration inequalities together
with suitable partitions of the sphere. These arguments can be extended to
general $p \geq 1$ by approximating $|t|^p$, on each cell away from the relevant
equator, by a polynomial of degree $O_{d,p}(\log(1/\eps))$. This yields
\[
	N_p(\ell_2^d, \eps) = \widetilde{O}_{d,p}(\eps^{-\frac{2(d-1)}{d+2p}}),
	\qquad N_p(d, \eps) = \widetilde{O}_{d,p}(\eps^{-\frac{2(d-1)}{d+2p-1}}).
\]
For $p \notin 2\mathbb{Z}$, the first bound is optimal up to logarithmic
factors (recalling \eqref{eqn:LB}), whereas the second is not, as its exponent
has denominator $d + 2p - 1$ instead of the optimal $d + 2p$. To the best of our knowledge,
these extensions have not been recorded in the literature. For completeness, we give the proofs in Appendices~\ref{sec:ball} and \ref{sec:general_bourgain}, respectively.

\subsection{Proof strategy}

Our proof is inspired by the discrepancy approach introduced by Matou\v{s}ek for
the case $p = 1$~\cite{mat96}, but the main mechanism is different. It combines
two ingredients: an approximation of $|t|^p$ by a polynomial with a remainder
of small total variation, and a partial-colouring argument which redistributes the mass among atoms to reduce the support size.
The polynomial part contributes no error, while the remaining error is controlled through equatorial band discrepancy.

Given a finitely supported uniform measure $\mu$ on $\Sph^{d-1}$, Matou\v{s}ek
pairs many of the points in $\supp(\mu)$ using a low-crossing matching
$e_j = \{a_j, b_j\}$ ($j \in [k]$).
The pairing is chosen to have a low crossing number, namely, every equator crosses only a few pairs.
He then chooses one point from each pair, reducing
the support size of the measure. The choice from each pair is encoded by a
colouring $\chi \in \{-1,1\}^k$. If $\chi_j = 1$, then $a_j$ is kept, otherwise
$b_j$ is kept. The problem is then reduced to choosing $\chi$ so that the
signed sums
\[
	\sum_{j = 1}^k \chi_j \left( |\langle u,a_j\rangle|-|\langle u,b_j\rangle
	| \right)
\]
are small uniformly over $u \in \Sph^{d-1}$. This is achieved through a multiscale
discrepancy construction for the absolute-value functions $y\mapsto |\langle u,
y\rangle|$.

A direct extension of this approach to $|\langle u,y\rangle|^p$ would control
the pair differences by the Lipschitz estimate
\[
	\left| |\langle u,a\rangle|^p - |\langle u,b\rangle|^p \right| \le p\|a - b
	\|_2,
\]
which gives essentially the same type of bound as in the $p = 1$ case and does not
yield the desired $d + 2p$ exponent. Instead, we approximate $|t|^p$ by
a polynomial with a remainder of small total variation. The polynomial part
is matched exactly after sparsification, while the approximation error from the remainder is
controlled through equatorial band discrepancy.

We write
\[
	|t|^p = q_K(t) + b_K(t),
\]
where $q_K$ is a polynomial of degree at most $K$ and $b_K$ is an even function with small total variation
\[
	\norm{b_K}_{\mathrm{TV}[-1,1]} \lesssim_p K^{-p}.
\]

We also consider fractional colourings. Let $\mu$ be a measure as in \eqref{eqn:mu}. By splitting heavy atoms if necessary, we assume that $\rho_i\lesssim 1/M$. Our goal is to find $x \in [-1,1]^M$ such that many coordinates of $x$ are $1$ or $-1$ and
\begin{equation}\label{eqn:exactness_intro}
	\sum_{i = 1}^M \rho_i x_i \phi(z_i) = 0 \qquad\text{for every }\phi \in
	\cP_K(\Sph^{d-1}),
\end{equation}
where $\cP_K(\Sph^{d-1})$ denotes the space of spherical polynomials of degree
at most $K$. 
Define
\[
	\nu_{+} = \sum_{i = 1}^M \rho_i(1+x_i)\delta_{z_i},\quad \nu_{-} = \sum_{i = 1}
	^M \rho_i(1-x_i)\delta_{z_i}.
\]
Taking $\phi \equiv 1$ in \eqref{eqn:exactness_intro} gives $\sum_i \rho_i x_i = 0$, so both $\nu_+$ and $\nu_-$ are probability measures. Let the new measure $\nu$ be the one with smaller support. Then
\[
	\int_{\Sph^{d-1}}q_K(\ip{u}{y}) \, d\nu(y) = \int_{\Sph^{d-1}}q_K(\ip{u}{y})\, d\mu(y)
\]
for all $u \in \Sph^{d-1}$. Thus the polynomial component cancels exactly; 
the remaining error comes from $b_K$ and we need to control
\[
	\int_{\Sph^{d-1}} b_K(\ip{u}{y}) \, d(\nu-\mu)(y).
\]
Write $b_K(t)=\beta_K(|t|)$. 
Since $\sum_{i=1}^M \rho_i x_i=0$, the boundary term vanishes in Stieltjes integration by parts, and for every $u\in\Sph^{d-1}$,
\[
	\left| \int_{\Sph^{d-1}} b_K(\ip{u}{y}) \, d(\nu - \mu)(y) \right|
	= \left| \int_0^1 \sum_{i=1}^M \rho_i x_i \mathbf{1}_{\cB(u,t)}(z_i)\,d\beta_K(t) \right|,
\]
where
\[
	\cB(u,t) = \{y \in \Sph^{d-1}:|\langle u,y\rangle| \le t\}
\]
are equatorial bands. It is tempting to upper bound this quantity simply by
\begin{equation}\label{eqn:naive_upper_bound}
	\left| \int_0^1 \sum_{i=1}^M \rho_i x_i \mathbf{1}_{\cB(u,t)}(z_i)\,d\beta_K(t) \right|
	\leq
	\norm{\beta_K}_{\mathrm{TV}[0,1]} \sup_{u\in \Sph^{d-1}, t \in [0,1]} \left|\sum_{i=1}^M \rho_i x_i \mathbf{1}_{\cB(u,t)}(z_i)\right|.
\end{equation}
The supremum is the discrepancy of equatorial bands. One can pair the points using the low-crossing matching theorem (see Theorem~\ref{thm:low-crossing-matching}). Weighted mass is
then transferred within each pair. Every equatorial band crosses
at most $O_d(M^{1-1/(d-1)})$ pairs. The exactness conditions are linear in
these pairwise perturbations and define a subspace of codimension at most
$\dim\cP_K(\Sph^{d-1})$. Moreover, the discrepancy vector of a band is sparse,
since it is supported only on the pairs crossed by that band. Rothvoss's
subspace partial colouring theorem~\cite{rothvoss} therefore produces a
perturbation that preserves exactness, controls all band discrepancies, and
saturates a constant fraction of the unsaturated variables. One can show that
\[
	\sup_{u\in \Sph^{d-1}, t\in [0,1]} \left|\sum_{i=1}^M \rho_i x_i \mathbf{1}_{\cB(u,t)}(z_i)\right| \lesssim_d M^{-\frac12-\frac{1}{2(d-1)}}\sqrt{\log(2M)}.
\]
Since $M$ becomes a constant factor smaller each time, one can repeat this process until $M\asymp_d K^{d-1}$. We stop at this scale because, below it, the codimension hypothesis required by the partial-colouring theorem is no longer guaranteed. Thus, if the final size is $N$ (so $K\asymp_d N^{1/(d-1)}$), the total error will be
\[
	O\left(K^{-p}\cdot N^{-\frac12-\frac{1}{2(d-1)}}\sqrt{\log(2N)}\right) = \tilde{O}(N^{-(d+2p)/(2(d-1))}).
\]
This is the correct power of $N$, but it introduces additional logarithmic factors.

The loss comes from the na\"ive bound \eqref{eqn:naive_upper_bound}, since it separates the total variation of $b_K$ from a uniform band discrepancy estimate, and the latter estimate has a logarithmic factor. To avoid this loss, we retain the Stieltjes integral throughout the partial colouring process and build the discrepancy vectors only after integrating against $d\beta_K$. Suppose that the unsaturated indices are paired by a low-crossing matching $\cE$. For $e = \{a,b\}\in\cE$, define
\[
	A_u(e) = \int_0^1 (\mathbf{1}_{\cB(u,t)}(z_a) - \mathbf{1}_{\cB(u,t)}(z_b)) \,d\beta_K(t).
\]
If $\theta_e$ is the amount of mass transferred along the pair $e$, then the corresponding increment in integrated discrepancy has the form
\[
	\sum_{e\in\cE} \theta_e A_u(e).
\]
Thus the problem is to choose a vector $\theta$ so as to control these sums uniformly over $u$, while preserving the exactness conditions.

The key change from the earlier approach is that a chaining argument is applied
to appropriate linear transformations of the integrated vectors
$\{A_u:u\in\Sph^{d-1}\}$. The packing lemma (see
Corollary~\ref{cor:packing}) provides the covering-number bounds required for
chaining. The chaining argument then constructs a symmetric convex set that has
sufficiently large Gaussian measure and enforces a uniform bound on integrated
discrepancy. Rothvoss's partial-colouring theorem applied to this set produces a
transfer vector $\theta$
that saturates a constant fraction of the unsaturated variables, preserves
polynomial exactness, and, at a stage with $M$ active atoms, satisfies
\[
	\sup_{u\in\Sph^{d-1}}
	\left|\sum_{e\in\cE}\theta_e A_u(e)\right|
	\lesssim_d 
	M^{-\frac12-\frac{1}{2(d-1)}} \norm{\beta_K}_{\mathrm{TV}[0,1]}.
\]
This removes the logarithmic factor in the na\"ive estimate and, after iteration, yields the desired tight bound.

\section{Ingredients for Sparsification with Additive Error}

\paragraph{Notation} In the remainder of this paper, $d \ge 2$ is fixed and $n = d - 1$. For $u \in \Sph^n$
and $t \in [0,1]$, define the equatorial band
\[
	\cB(u,t) = \{y \in \Sph^n:|\ip{u}{y}| \le t\}.
\]

Let $\cP_K(\Sph^n)$ denote the restrictions to $\Sph^n$ of real polynomials
on $\R^{n+1}$ of degree at most $K$.

For an indexed family $Y = (y_1,\dots,y_m)$ of points in $\Sph^n$ and a linear space $V$ of real-valued functions, define
\[
	V|_Y = \{(q(y_1),\dots,q(y_m)): q \in V\} \subseteq \R^m.
\]
If $I \subseteq [m]$, we write $Y_I = (y_i)_{i \in I}$ and $V|_I := V|_{Y_I}$.

Let $\mu$ and $\nu$ be finite measures on $\Sph^n$, and let
$V \subset C(\Sph^n)$. We say that $\mu$ and $\nu$ are exact on $V$ if
\[
	\int_{\Sph^n} q(y)\,d\nu(y) = \int_{\Sph^n} q(y)\,d\mu(y) \qquad \forall q \in V.
\]

When writing a finitely supported measure $\mu$ as
\[
	\mu = \sum_{i = 1}^m \alpha_i \delta_{y_i},
\]
we allow repetitions among the points $y_i$, and $m$ counts atoms with multiplicity. Such a representation is called \(B\)-balanced if \(\alpha_i\le B/m\) for every \(i\).

Throughout, $C_n, C_{n,p}, \dots$ denote positive constants depending only on the indicated parameters;
unless explicitly fixed in a theorem or lemma statement, their values may change from one occurrence to the next.

\subsection{Sparsification with band discrepancy control}

The proof of our sparsification result is largely reduced to the following discrepancy estimate.
It replaces a finitely supported probability measure by another, supported on fewer points,
that is exact on a prescribed finite-dimensional space, while controlling discrepancies for equatorial bands. Its proof is given in \Cref{sec:sparsification-proof}.

\begin{theorem}[Sparsification with integrated band discrepancy control]
	\label{input:finite-band-sampling} Let $n \ge 1$ be fixed. There exist
	constants $c_n, C_n > 0$ such that the following holds.

	Let $N \ge 2$ be an integer, and let
	\[
		\mu = \sum_{i = 1}^M \rho_i \delta_{z_i}, \qquad \rho_i > 0,\quad \sum_{i = 1}^M \rho_i = 1, \quad z_i \in \Sph^n
	\]
	be a finitely supported probability measure on \(\Sph^n\). Let \(Z = (z_1,\dots,z_M)\) be the indexed family associated with $\mu$, and \(V\subset C(\Sph^n)\) be a finite-dimensional linear subspace containing the constants such that
	\[
		\dim V|_Z \le c_n N.
	\]
	Let $\xi$ be a finite signed Borel measure on $[0,1]$.
	Then there exists a positive atomic probability measure
	\[
		\nu = \sum_{i = 1}^{N'}w_i\delta_{y_i}, \qquad N' \le N,\quad y_i \in \supp
		(\mu),\quad w_i > 0,
	\]
	such that $\nu$ and $\mu$ are exact on $V$, that is,
	\[
		\int_{\Sph^n}q(y)\,d\nu(y) = \int_{\Sph^n}q(y)\,d\mu(y) \qquad \forall q \in V
		,
	\]
	and
	\[
	\sup_{u \in \Sph^n}\left|\int_0^1\bigl(\nu(\cB(u,t)) - \mu(\cB(u,t))\bigr)\,d\xi(t)\right| \le
		C_n N^{-1/2-1/(2n)}\norm{\xi}_{\mathrm{TV}}.
	\]
\end{theorem}

\subsection{\texorpdfstring{Polynomial approximation of $|t|^p$}{Polynomial approximation}}

The following is our main polynomial-approximation lemma.
\begin{lemma}\label{lem:polynomial_approx}
Let $p\geq 1$. There exist constants $C_p, K_p > 0$ such that for every integer $K\geq K_p$, there exists an even polynomial $p_K$ of degree at most $K$ such that $p_K(0) = 0$ and
\[
	\bigl\| |t|^p - p_K(t) \bigr\|_{\mathrm{TV}[-1,1]} \leq C_p K^{-p}.
\]
\end{lemma}

The proof of the lemma is based on the following Jackson-type approximation result in $L_1$-norm, whose proof is deferred to Appendix~\ref{sec:L_1-approx}.
\begin{lemma}\label{lem:important_approx}
	Let $\alpha \ge 0$. There exist constants $C_\alpha, K_\alpha > 0$ such that for every integer $K\geq K_\alpha$, there exists an odd polynomial $p_K$ of degree at most $K$ such that
	\[
		\int_{-1}^1 \left| \sgn(t)|t|^\alpha - p_K(t) \right| \, dt \leq C_\alpha K^{-\alpha-1}.
	\]
\end{lemma}

Now we prove Lemma~\ref{lem:polynomial_approx}.

\begin{proof}[Proof of Lemma~\ref{lem:polynomial_approx}]
	Let $\alpha = p-1$. By Lemma~\ref{lem:important_approx}, there exist constants $C_p, K_p > 0$ such that for every $K\geq K_p$, there exists an odd polynomial $q_{K-1}$ of degree at most $K-1$ such that
	\[
		\int_{-1}^1 \left| \sgn(t)|t|^\alpha - q_{K-1}(t) \right| \, dt \leq C_p K^{-p}.
	\]
	Suppose that $K\geq K_p+1$ and define
	\[
		p_K(t) = p \int_0^t q_{K-1}(s) \, ds.
	\]
	Since $q_{K-1}$ is odd, $p_K$ is even. It is also clear that $p_K(0) = 0$ and $\deg(p_K)\leq K$. The function $|t|^p - p_K(t)$ is absolutely continuous on $[-1,1]$, therefore
	\begin{align*}
		\bigl\| |t|^p - p_K(t) \bigr\|_{\mathrm{TV}[-1,1]} 
		&\leq \int_{-1}^1 \left| \frac{d}{dt}(|t|^p - p_K(t)) \right|\,dt\\
		&= p \int_{-1}^1 \left| \sgn(t)|t|^{p-1} - q_{K-1}(t) \right|\,dt \\
		&\leq C_p K^{-p}. \qedhere
	\end{align*}
\end{proof}

\section{Proof of Sparsification with Additive Error}

\begin{proof}[Proof of Theorem~\ref{thm:additive}]
	Let $n = d - 1$.
	Choose an integer $N$, to be fixed later. By increasing the hidden constant in
	the final support bound, we may assume $N \ge N_0(d,p)$. The standard dimension
	formula for spherical polynomials (see~\cite[Corollary 1.1.5]{DX13}) gives
	\[
		\dim\cP_K(\Sph^n) = \binom{n+K}{n} + \binom{n+K-1}{n} = O_n(K^n).
	\]
	Choose $\gamma_n > 0$ sufficiently small and set
	\[
		K = \left\lfloor\gamma_n N^{1/n}\right\rfloor.
	\]
	By increasing $N_0(d,p)$ if necessary, we have $K \ge K_p$, and the choice
	of $\gamma_n$ ensures that
	\[
		\dim\cP_K(\Sph^n) \le c_n N,
	\]
	where $c_n$ is the constant in Theorem~\ref{input:finite-band-sampling}.

	Let $p_K$ be the polynomial given by Lemma~\ref{lem:polynomial_approx}, and set
	\[
		b(t)=|t|^p-p_K(t)\quad (t\in[-1,1]).
	\]
	Write $b(t)=\beta(|t|)$ for $t\in[-1,1]$. Then the associated finite signed
	Stieltjes measure $d\beta$ on $[0,1]$ satisfies
	\[
		\norm{d\beta}_{\mathrm{TV}} = \frac{1}{2}\norm{b}_{\mathrm{TV}[-1,1]} \le C_p K^{-p}.
	\]

	Apply Theorem~\ref{input:finite-band-sampling} to the finitely supported measure
	$\mu$ with $V = \cP_K(\Sph^n)$ and the signed measure $d\beta$. We obtain a positive atomic probability
	measure
	\[
		\nu = \sum_{i = 1}^{N'}w_i\delta_{y_i}, \qquad N' \le N,\quad y_i \in \supp(\mu
		),
	\]
	such that
	\[
		\int_{\Sph^n}Q(y)\,d\nu(y) = \int_{\Sph^n}Q(y)\,d\mu(y) \qquad \forall Q \in
		\cP_K(\Sph^n),
	\]
	and
	\[
	\Delta := \sup_{u \in \Sph^n}\left|\int_0^1\bigl(\nu(\cB(u,t)) - \mu(\cB(u,t
		))\bigr)\,d\beta(t) \right| \le C_{n,p} K^{-p} N^{-1/2-1/(2n)}.
	\]

	Let
	\[
		\lambda = \nu - \mu.
	\]
	Then $\lambda(\Sph^n) = 0$.

	Fix $u \in \Sph^n$ and set $Q_{u,K}(y)=p_K(\ip{u}{y})$. Then
	$Q_{u,K} \in \cP_K(\Sph^n)$.

	Since
	\[
		\int_{\Sph^n} Q_{u,K}(y)\,d\lambda(y) = 0,
	\]
	and $\lambda(\Sph^n)=0$, Stieltjes integration by parts gives
	\[
		\int_{\Sph^n} b(\ip{u}{y})\,d\lambda(y)
		= -\int_0^1 \lambda(\cB(u,t))\,d\beta(t).
	\]
	Hence
	\[
		\left| \int_{\Sph^n} |\ip{u}{y}|^p\,d\lambda(y) \right| = \left| \int_{\Sph^n} b(\ip{u}{y})\,d\lambda(y) \right| 
		\le \Delta.
	\]
	Therefore
	\[
	\sup_{u \in \Sph^n}\left| \int_{\Sph^n} |\ip{u}{y}|^p\,d\nu(y) - \int_{\Sph^n} |\ip{u}{y}|^p
		\,d\mu(y) \right| \le C_{d,p} K^{-p} N^{-1/2-1/(2n)}.
	\]
	Since
	\[
		K \asymp_n N^{1/n},
	\]
	we obtain
	\[
	\sup_{u \in \Sph^n}\left| \int_{\Sph^n} |\ip{u}{y}|^p\,d\nu(y) - \int_{\Sph^n} |\ip{u}{y}|^p
		\,d\mu(y) \right|
		\leq C_{d,p}N^{-\frac{d+2p}{2(d-1)}}.
	\]
	Consequently, choosing
	\[
		N\asymp_{d,p}
		\eps^{-\frac{2(d-1)}{d+2p}}
	\]
	with a sufficiently large implicit constant gives 
	\[
		\sup_{u \in \Sph^n}\left| \int_{\Sph^n} |\ip{u}{y}|^p\,d\nu(y) - \int_{\Sph^n} |\ip{u}{y}|^p
		\,d\mu(y) \right| \le \eps. \qedhere
	\]
\end{proof}

\section{Sparsification with Band Discrepancy Control}
\label{sec:sparsification-proof}

In this section, we prove Theorem~\ref{input:finite-band-sampling}. The proof
relies on two principal external results: the low-crossing matching theorem (see, e.g.~\cite[Theorem~5.17]{matousek-discrepancy}) and Rothvoss's partial-colouring theorem~\cite{rothvoss}.

\subsection{Crossing numbers}

Let $(X,\mathcal{R})$ be a set system. Its dual shatter function is
\[
	\pi^{*}_{\mathcal{R}}(m) = \max_{\substack{\mathcal{A} \subseteq \mathcal{R}\\ |\mathcal{A}| \le m}}
	\left| \left\{ \{R \in \mathcal{A}:x \in R\}:x \in X \right\} \right|.
\]
Equivalently, $\pi^{*}_{\mathcal{R}}(m)$ is the maximum number of membership
patterns induced on points of $X$ by $m$ ranges.
A matching on $X$ is a collection of pairwise disjoint two-element subsets of
$X$; its elements are called edges.
We say that a range $R\in\mathcal R$ crosses an edge $\{x,y\}$ if
$|R\cap\{x,y\}|=1$.

We use the following classical low-crossing result.

\begin{theorem}[Low-crossing matching {\cite[Theorem~5.17]{matousek-discrepancy}}]
	\label{thm:low-crossing-matching}
	Let $(X,\mathcal{R})$ be a finite set system with $|X| = r$, and suppose that,
	for some constants $C_0 \geq 1$ and $q > 1$,
	\[
		\pi^{*}_{\mathcal{R}}(m) \le C_0 m^q \qquad \forall m \ge 1 .
	\]
	Then there is a matching $\cE$ on $X$ of size $\lfloor r/2\rfloor$ such that
	every $R\in\mathcal R$ crosses at most $C_{q,C_0}r^{1-1/q}$ edges of $\cE$.
\end{theorem}

\begin{proof}
For even $r$, this is precisely the cited theorem. For odd $r$, remove one
element of $X$, apply the theorem to the remaining set, and leave the removed
element unmatched.
\end{proof}

\subsection{Band arrangements}

For $u \in \Sph^n$ and $t \in [0,1]$, recall that
\[
	\cB(u,t) = \{y \in \Sph^n:|\langle u,y\rangle| \le t\}.
\]
Let
\[
	\mathfrak{B}_n = \{\cB(u,t):u \in \Sph^n,\ t \in [0,1]\}.
\]

\begin{lemma}[Dual shatter bound for bands]
	\label{lem:spherical-band-dual-shatter} For every fixed $n \ge 1$, there is a
	constant $C_n > 0$ such that
	\[
		\pi^{*}_{\mathfrak{B}_n}(m) \le C_n m^n \qquad \forall m \ge 1.
	\]
\end{lemma}

Lemma~\ref{lem:spherical-band-dual-shatter} follows from the standard bound for
polynomial sign patterns; see \cite[Section 6.2]{matousek-book}. Indeed, after
stereographic projection, the boundaries of the $m$ bands are represented by $2
m$ quadratic polynomials in $n$ variables, and hence determine $O_n(m^n)$
membership patterns.

\subsection{Packing Lemma}

For a set system $(X,\cR)$, its primal shatter function is
\[
	\pi_{\cR}(m) = \max_{\substack{A\subseteq X\\ |A|\leq m}} \left|\{R\cap A: R\in\cR\}\right|.
\]

We shall need the following probability-space version of the packing lemma, which is in the same spirit as the classical packing lemma for finite set systems. Its proof is deferred to Appendix~\ref{sec:packing}.

\begin{lemma}[Packing lemma for probability spaces]\label{lem:packing}
	Let $(\Omega,\cM,\mu)$ be a probability space and $\cR\subseteq \cM$ be a set system. Define the pseudometric on $\cR$ as
	\[
		\rho(R_1, R_2) = \mu(R_1 \triangle R_2).
	\]
	Suppose that there exist constants $C_0 > 0$ and $q>1$ such that
	\[
		\pi_{\cR}(m) \leq C_0 m^q, \qquad \forall m\geq 1.
	\]
	Then there exists a constant $C_{q, C_0} > 0$ such that for every $\delta\in (0,1]$, every $\delta$-separated family $\cP\subseteq \cR$ satisfies 
	\[
		|\cP| \leq C_{q, C_0} \delta^{-q}.
	\]
\end{lemma}

\begin{corollary}\label{cor:packing}
	Let $(\Omega,\cM,\mu)$ be a probability space and let $S$ be a finite set with $|S|\geq 2$. Let $\cF$ be a family of measurable functions $f:\Omega\to S$. Define
	\[
		\rho(f,g) = \mu(\{\omega\in\Omega: f(\omega)\neq g(\omega)\}).
	\]
	Suppose there exist constants $C_0 > 0$ and $q>1$ such that, for every $m\geq 1$,
	\[
		\sup_{\omega_1,\dots,\omega_m\in \Omega}|\{ (f(\omega_1),\dots,f(\omega_m)): f\in \cF \}| \leq C_0 m^q,
	\]
	then every $\delta$-separated family $\cP\subseteq \cF$ satisfies 
	\[
		|\cP| \leq C_{q, C_0} \left(\frac{|S|}{\delta}\right)^{q}.
	\]
\end{corollary}
\begin{proof}
	Endow $\tilde{\Omega} = \Omega\times S$ with probability measure $\tilde{\mu} = \mu \otimes \lambda_S$, where $\lambda_S$ denotes the uniform measure on $S$.
	Consider the graph $\Gamma_f = \{(\omega,s): s = f(\omega)\}$. Let $\cR = \{\Gamma_f: f\in \cF\}$ be the set system in $\tilde{\Omega}$. It is not difficult to see that
	\[
		\pi_{\cR}(m) \leq \sup_{\omega_1,\dots,\omega_m\in \Omega}|\{ (f(\omega_1),\dots,f(\omega_m)): f\in \cF \}|	\leq C_0 m^q.
	\]
	Next we consider the metric on $\tilde{\Omega}$. For $f,g\in\cF$ and $\omega\in\Omega$,
	\[
		|\{s\in S: (\omega,s)\in \Gamma_f\triangle\Gamma_g \}| 
		= \begin{cases}
			0, & f(\omega) = g(\omega) \\
			2, & f(\omega) \neq g(\omega)
		  \end{cases}
	\]
	Thus
	\[
		\tilde{\mu}(\Gamma_f\triangle \Gamma_g) = \frac{2}{|S|}\rho(f,g).	
	\]
	If $\cP$ is a $\delta$-separated family of $\cF$, then $\tilde{\cP} = \{\Gamma_f: f\in\cP\}$ is a $(2\delta/|S|)$-separated family of $\cR$ and
	\[
		|\cP| = |\tilde{\cP}| \leq C_{q,C_0}\left(\frac{|S|}{2\delta}\right)^q.
	\]
\end{proof}

\subsection{Chaining argument}

We need \v{S}id\'ak's lemma~\cite{sidak67}, which states that if
$Y_1,\dots,Y_m$ are centred jointly Gaussian random variables and
$a_1,\dots,a_m\geq 0$, then
\[
	\Pr\left\{|Y_i|\leq a_i\text{ for all }i\in[m]\right\} \geq \prod_{i=1}^m \Pr\{|Y_i|\leq a_i\}.
\]

\begin{lemma}\label{lem:chaining}
	Let $F$ be a finite-dimensional Euclidean space, let $q\geq 1$ and $C_0\geq 1$,
	and suppose that $0<R\leq L$. Let $X\subseteq B(0,R)$ be a bounded set whose
	covering numbers satisfy
	\[
		N(X,\norm{\cdot}_2,\eps)
		\leq C_0\left(\frac{L}{\eps}\right)^q,
		\qquad 0<\eps\leq L.
	\]
	Then, for every $s\geq C_0$, there exists a symmetric convex set
	$K_s\subseteq F$ such that
	\[
		\gamma_F(K_s)\geq e^{-s}
	\]
	and
	\[
		\sup_{v\in X}|\ip{v}{h}|
		\leq C_q\max\left\{R,L\left(\frac{C_0}{s}\right)^{1/q}\right\},
		\qquad \forall h\in K_s.
	\]
	Here $\gamma_F$ denotes the standard Gaussian measure on $F$.
\end{lemma}

\begin{proof}[Proof of Lemma~\ref{lem:chaining}]
	Set
	\[
		\widetilde{R} = \max \left\{R, L\left(\frac{C_0}{s}\right)^{1/q} \right\}
		\qquad\text{and}\qquad
		M = C_0\left(\frac{L}{\widetilde{R}}\right)^q.
	\]
	Since $s\geq C_0$, we have $R\leq \widetilde{R}\leq L$ and $M\leq s$.
	For each $j\geq1$, choose a $(2^{-j/2}\widetilde{R})$-net $\cN_j$ of
	$X$ such that
	\[
		|\cN_j| \leq M2^{qj/2},
	\]
	and let $\cN_0  =\{0\}$. For each $j\geq 1$, choose
	$\pi_j: \cN_j \to \cN_{j-1}$ such that
	\[
		\norm{x-\pi_j(x)}_2\leq 2^{-(j-1)/2} \widetilde{R}.
	\]
	For $j=1$, this follows from $X\subseteq B(0,R)\subseteq B(0,\widetilde{R})$.

	Choose $A_q > 0$ sufficiently large that
	\[
		c_q := 2^{q/2}e^{-A_q^2/2} \leq \frac15.
	\]
	Since $q \geq 1$, this also gives $2e^{-A_q^2/2} < 1/2$ and
	$4c_q/(1-c_q)\leq 1$.
	Define
	\[
		K_s = \bigcap_{j\geq 1} \bigcap_{x\in\cN_j}	
				\left\{ h\in F: |\ip{x-\pi_j(x)}{h}| \leq A_q \sqrt{j} \norm{x-\pi_j(x)}_2 \right\},
	\]
	and let $K_{s,J}$ denote the intersection restricted to $1\leq j\leq J$.
	If $g_F$ is a standard Gaussian vector in $F$, then \v{S}id\'ak's lemma gives
	\[
		\gamma_F(K_{s,J})
		\geq \prod_{j=1}^J\left(1-2e^{-A_q^2j/2}\right)^{|\cN_j|}
		\geq \exp\left(-4M\sum_{j=1}^Jc_q^j\right).
	\]
	Letting $J\to\infty$,
	\[
		\gamma_F(K_s)\geq\exp\left(-4M\frac{c_q}{1-c_q}\right)
		\geq e^{-M}\geq e^{-s}.
	\]

	Fix $v\in X$ and choose $v_J \in \cN_J$ such that
	$\norm{v-v_J}_2 \leq 2^{-J/2}\widetilde{R}$. For $j\leq J$, let
	$v_{j-1}=\pi_j(v_j)$. For $h\in K_s$,
	\begin{equation}\label{eqn:inner-product-bound_chaining}
		|\ip{v}{h}| 
		\leq 2^{-J/2}\widetilde{R}\norm{h}_2 + A_q\widetilde{R}\sum_{j=1}^J\sqrt{j}\,2^{-(j-1)/2}
		\leq 2^{-J/2}\widetilde{R}\norm{h}_2 + C A_q \widetilde{R}.
	\end{equation}
	Letting $J\to\infty$ completes the proof.
\end{proof}

\subsection{Fractional colouring under exactness constraints}
We need the following partial-colouring theorem by Rothvoss~\cite{rothvoss}.

\begin{theorem}[Subspace convex set partial colouring, {\cite[Lemma 9]{rothvoss}}]
	\label{thm:subspace-convex-set-pc} There exist absolute constants $c_0,\kappa
	_0 > 0$ with the following property. Let $F \subseteq \mathbb{R}^k$ be a
	linear subspace with
	\[
		\operatorname{codim}(F) \le c_0 k,
	\]
	and $K \subseteq F$ be a centrally symmetric convex set such that
	\[
		\gamma_F(K) \ge \exp(-c_0 k),
	\]
	where $\gamma_F$ denotes the standard Gaussian measure on $F$. Then for every
	$z^0 \in (-1,1)^k$, there exists
	\[
		z \in (z^0 + K)\cap[-1,1]^k
	\]
	such that at least $\kappa_0 k$ coordinates of $z$ belong to $\{-1,+1\}$.
\end{theorem}

The following two lemmata address the main technical difficulties in the proof.
\begin{lemma}
	\label{lem:integrated-band-body}
	Let $n\geq 1$, $q>1$, $C_0>0$, and $\tau>0$. Let $\cE$ be a matching of size $k\geq 1$ on an
	index set, and let $y_i\in \Sph^n$ be given for each index appearing in
	$\cE$. Let $0\leq\Lambda\leq k$ and suppose that every equatorial band crosses at most $\Lambda$ edges of
	$\cE$. For $e=\{a,b\}\in\cE$ and $u\in\Sph^n$, set
	\[
		a_u(e,t)=\mathbf{1}_{\cB(u,t)}(y_a)-\mathbf{1}_{\cB(u,t)}(y_b).
	\]
	Let $\xi$ be a finite signed Borel measure on $[0,1]$, and define
	$A_u\in\R^k$ by
	\[
		A_u(e)=\int_0^1 a_u(e,t)\,d\xi(t).
	\]
	Suppose that the family of functions $(e,t)\mapsto a_u(e,t)$, indexed by
	$u\in\Sph^n$, satisfies
	\[
		\sup_{(e_1,t_1),\dots,(e_\ell,t_\ell)}
		\left|\left\{\bigl(a_u(e_1,t_1),\dots,a_u(e_\ell,t_\ell)\bigr):u\in\Sph^n\right\}\right|
		\leq C_0\ell^q, \qquad \ell\geq 1.
	\]
	Let $D$ be a diagonal matrix with $0\le D_{ee}\le \tau$, and let
	$F\subseteq\R^k$ be a linear subspace. Set
	\[
		R = \tau\sqrt{\Lambda}\norm{\xi}_{\mathrm{TV}}.
	\]
	There is a constant $C_{q,C_0}\geq1$, depending only on $q$ and $C_0$, such
	that, for every $s\geq C_{q,C_0}$, there exists a symmetric convex set
	$K_s\subseteq F$ such that
	\[
		\gamma_F(K_s)\geq e^{-s}
	\]
	and, for every $h\in K_s$,
	\[
		\sup_{u\in\Sph^n}|\ip{P_F D A_u}{h}| 
		\leq 
		C_{q,C_0}\max\left\{R,	\tau\sqrt{k}\norm{\xi}_{\mathrm{TV}}s^{-1/(2q)}\right\},
	\]
	where $P_F$ denotes the orthogonal projection onto $F$.
\end{lemma}

\begin{proof}
	If $\norm{\xi}_{\mathrm{TV}}=0$, take $K_s=F$. Otherwise define a probability
	measure $\overline{\xi}$ by
	\[
		\overline{\xi}(E) = \frac{|\xi|(E)}{\norm{\xi}_{\mathrm{TV}}},
	\]
	where $|\xi|$ denotes the total variation measure of $\xi$. Also define the probability space
	\[
		\Omega = \cE \times [0,1]
	\]
	with probability measure
	\[
		\mathbb{P}(E) = \frac{1}{k}\sum_{e\in\cE} \int_0^1 \mathbf{1}_{E}(e,t) \, d\overline{\xi}(t).
	\]
	On $\Sph^n$ define a pseudometric
	\[
		\rho(u,v) = \mathbb{P}(\{(e,t): a_u(e,t)\neq a_v(e,t) \}).
	\]
	For $c,c'\in \{-1,0,1\}$ we have
	\[
		|c-c'|^2 \leq 4\cdot \mathbf{1}_{\{c\neq c'\}},
	\]
	and hence
	\[
		\norm{A_u-A_v}_2 \leq 2\sqrt{k \rho(u,v)} \norm{\xi}_{\mathrm{TV}}.
	\]
	By Corollary~\ref{cor:packing}, there is a constant $C'_{q,C_0}$, depending
	only on $q$ and $C_0$, such that the covering numbers of these functions in the
	pseudometric $\rho$ are at most $C'_{q,C_0}\delta^{-q}$ for $0<\delta\leq 1$.

	Now we upper bound the covering numbers of
	\[
		\cA = \{P_F D A_u: u \in \Sph^n\}\subseteq F.
	\]
	By a similar calculation to the above,
	\[
		\norm{DA_u}_2^2 \leq  \tau^2 \norm{\xi}_{\mathrm{TV}}^2 \int_0^1 \sum_{e\in\cE} |a_{u}(e,t)|^2 \, d\overline{\xi}(t) \leq \tau^2 \norm{\xi}_{\mathrm{TV}}^2 \Lambda,
	\]
	and therefore $\norm{P_F D A_u}_2\le R$. Thus $\cA\subseteq B(0,R)$.
	If $R = 0$, the conclusion follows by taking $K_s = F$. Hence assume $R > 0$ and set
	\[
		L = 2\tau\sqrt{k}\norm{\xi}_{\mathrm{TV}}.
	\]
	Since $\Lambda\leq k$, we have $R\leq L$. Fix $0 < \eps\leq L$ and
	let $\delta = (\eps/L)^2$. Then $0 < \delta \leq 1$, and when $\rho(u,v) < \delta$,
	\[
		\norm{P_F D (A_u-A_v)}_2
		\leq 2\tau\sqrt{k\delta}\norm{\xi}_{\mathrm{TV}}
		= \eps.
	\]
	The packing estimate therefore gives
	\[
		N(\cA, \norm{\cdot}_2, \eps) \leq C'_{q,C_0} \left(\frac{L}{\eps}\right)^{2q}.
	\]
	Lemma~\ref{lem:chaining}, applied with exponent $2q$, now yields a symmetric
	convex set $K_s\subseteq F$ with $\gamma_F(K_s)\geq e^{-s}$ and
	\[
		\sup_{u\in\Sph^n} |\ip{P_F D A_u}{h}| 
		\leq C_{2q} \max \left\{ R, L\left(\frac{C'_{q,C_0}}{s}\right)^{1/(2q)} \right\}
		\leq C_{q,C_0}\max \left\{ R, \tau\sqrt{k}\norm{\xi}_{\mathrm{TV}}s^{-1/(2q)} \right\}
	\]
	for every $h\in K_s$, after choosing $C_{q,C_0}$ sufficiently large in terms
	of $C'_{q,C_0}$ and $C_{2q}$.
\end{proof}

\begin{lemma}
	\label{lem:weighted-fractional-band-colouring}
	Let $n \geq 1$. There exist an absolute constant $\kappa > 0$ and constants
	$\bar{c}_n,r_n>0$ such that the following holds.

	Let $B \geq 1$ be a constant,
	\[
		\eta = \sum_{i = 1}^m \alpha_i\delta_{y_i}
	\]
	be a $B$-balanced representation of a probability measure on $\Sph^n$ and
	$V \subset C(\Sph^n)$ be a finite-dimensional linear space containing
	the constants. Let $x \in [-1,1]^m$, $I = \{ i: |x_i| < 1 \}$
	be the set of unsaturated coordinates and $r = |I|$. Suppose
	\[
		\dim V|_I \le \bar c_n r \qquad\text{and}\qquad r \geq r_n.
	\]
	Let $\xi$ be a finite signed Borel measure on $[0,1]$. Then there exists $x' \in [-1,1]^m$ such that:

	\begin{enumerate}[itemsep = 0pt, label = (\roman{*})]
		\item $x_i' = x_i$ for all $i \notin I$;

		\item at least $\kappa r$ coordinates in $I$ are saturated by $x'$, i.e., satisfy $|x_i'| = 1$;

		\item exactness on $V$ is preserved:
			\[
				\sum_{i = 1}^m \alpha_i (x_i'-x_i) q(y_i) = 0,\quad \forall q \in V;
			\]

		\item the integrated equatorial band discrepancy satisfies
			\[
				\sup_{u \in \Sph^n} \left|\int_0^1 \sum_{i = 1}^m \alpha_i(x_i'-x_i)\mathbf{1}_{\cB(u,t)}(y_i)\, d\xi(t) \right| \le C_n Bm^{-1} r^{1/2-1/(2n)} \norm{\xi}_{\mathrm{TV}}.
			\]
	\end{enumerate}
\end{lemma}

\begin{proof}
	We shall first choose a pairing of the unsaturated indices such that every equatorial
	band crosses few of the pairs, and then perturb the two coordinates of $x$ in
	each pair in opposite directions, choosing the perturbations by applying the
	partial-colouring theorem inside the subspace enforcing exactness.

	\paragraph{Low-crossing pairing.}
	Let
	\[
		\mathcal{R}_I = \bigl\{\{i \in I:y_i \in \cB(u,t)\}:u \in \Sph^n,\ t \in [0,1]\bigr\}.
	\]
	When $n=1$, order the indices $i_1,\dots,i_r$ of $I$ such that the corresponding points are ordered cyclically on $\Sph^1$, breaking ties arbitrarily. 
	Every equatorial band is a union of at most two arcs and therefore crosses at most four edges of this order.
	Pair consecutive indices, leaving the last index unmatched if $r$ is odd, and
	let $\cE$ be the resulting matching on the index set $I$.
	When $n>1$, by Lemma~\ref{lem:spherical-band-dual-shatter},
	$\pi^{*}_{\mathcal{R}_I}(\ell) \le C_n \ell^n$ for all $\ell \ge 1$, and hence
	Theorem~\ref{thm:low-crossing-matching} applied to $(I,\mathcal{R}_I)$ with
	$q=n$ yields a matching $\cE$ on the index set $I$ of size $\lfloor r/2\rfloor$ such that every
	equatorial band crosses at most $C_n r^{1-1/n}$ of its edges. Thus the
	crossing number satisfies
	\begin{equation}\label{eqn:low-crossing-lambda}
		\Lambda \le
		\begin{cases}
			C_n r^{1-1/n}, & n > 1, \\
			4,         & n = 1.
		\end{cases}
	\end{equation}
	In either case, write $\cE=\{e_1,\dots,e_k\}\subseteq \binom{I}{2}$, where
	$k=\lfloor r/2\rfloor$; if $r$ is odd, the unmatched coordinate is left unchanged.

	\paragraph{Pair perturbations and constraints.}
	For $\theta = (\theta_e)_{e \in \cE} \in \R^k$, define a perturbed vector $x(\theta) \in \R^m$ by
	\[
		x_a(\theta) = x_a + \frac{\theta_e}{\alpha_a}, \qquad
		x_b(\theta) = x_b - \frac{\theta_e}{\alpha_b}
	\]
	for each pair $e = \{a,b\} \in \cE$ and leaving all other coordinates unchanged. Then
	$x(\theta) \in [-1,1]^m$ if and only if $\theta_e \in J_e$ for every
	$e = \{a,b\}$, where
	\[
		J_e = \left[ \max\{-\alpha_a(1+x_a),\,\alpha_b(x_b-1)\},
			   \min\{\alpha_a(1-x_a),\,\alpha_b(1+x_b)\} \right].
	\]
	Since $a,b \in I$, the interval $J_e$ is nondegenerate and contains $0$ in its
	interior. Also, from $\alpha_i \le B/m$,
	\begin{equation}\label{eqn:tau}	
		J_e \subseteq [-\tau, \tau], \quad\text{where}\quad \tau := \frac{2B}{m}.
	\end{equation}

	Choose $q_1,\dots,q_s \in V$ whose restrictions to $\{y_i:i \in I\}$ span
	$V|_I$. Then
	\[
		s = \dim V|_I \le \bar c_n r.
	\]
	For $e = \{a,b\} \in \cE$, define
	\[
		b_{\ell}(e) = q_{\ell}(y_a) - q_{\ell}(y_b), \qquad \ell = 1,\dots,s.
	\]
	The following are the linear constraints which will enforce exactness on $V$:
	\[
		\sum_{e \in \cE}\theta_e b_{\ell}(e) = 0, \qquad \ell = 1,\dots,s.
	\]

	For $u\in \Sph^n$, define
	\[
		a_{u}(e,t) = \mathbf{1}_{\cB(u,t)}(y_a) - \mathbf{1}_{\cB(u,t)}(y_b), \qquad e = \{a,b\} \in \cE,
	\]
	and define the integral vector $A_u\in \R^k$ by
	\[
		A_u(e) = \int_0^1 a_{u}(e,t) \, d\xi(t).
	\]
	Thus,
	\[
		\int_0^1 \sum_{i = 1}^m \alpha_i(x_i(\theta) - x_i)\mathbf{1}_{\cB(u,t)}(y_i)\, d\xi(t) = \sum_{e\in\cE} \theta_e A_u(e).
	\]

	\paragraph{Partial colouring.}
	Normalize the box $\prod_{e \in \cE}J_e$ to $[-1,1]^k$.
	Write $J_e = [-L_e,R_e]$. Define a diagonal matrix $D \in \R^{k\times k}$ and a vector $z^0 \in \R^k$ by
	\[
		D_{ee} = \frac{L_e+R_e}{2}, \qquad
		z_e^0 = \frac{L_e-R_e}{L_e+R_e}.
	\]
	Then $z^0 \in (-1,1)^k$, and $\theta(z) = D(z-z^0)$ maps $[-1,1]^k$ onto
	$\prod_{e \in \cE}J_e$. Moreover, $(\theta(z))_e$ is an endpoint of $J_e$ whenever
	$z_e \in \{-1,+1\}$.

	Let
	\begin{equation}\label{eqn:F}
		F = \{h \in \mathbb{R}^k: \langle D b_{\ell},h\rangle = 0
		\text{ for all }\ell = 1,\dots,s\}.
	\end{equation}
	Choosing $\bar c_n > 0$ sufficiently small gives
	\[
		\operatorname{codim}(F) \le s \le c_0 k,
	\]
	where $c_0$ is the constant in the statement of Theorem~\ref{thm:subspace-convex-set-pc}.
	
	If $n=1$ or $\norm{\xi}_{\mathrm{TV}} = 0$, set $K=F$ and then $\gamma_F(K)=1$.
	If $n>1$ and $\norm{\xi}_{\mathrm{TV}}>0$, then $D_{ee}\leq \tau$. The same sign-pattern argument as in
	Lemma~\ref{lem:spherical-band-dual-shatter} shows that the functions
	$(e,t)\mapsto a_u(e,t)$ satisfy the shatter bound in
	Lemma~\ref{lem:integrated-band-body} with $q=n$ and some constant
	$C'_n$ depending only on $n$. With $q=n$ and $C_0=C'_n$, the constant
	$C_{q,C_0}$ supplied by that lemma depends only on $n$; we write it as
	$C_n$.

	Increasing $r_n$ if necessary, we may assume that $c_0k\geq C_n$.
	Apply Lemma~\ref{lem:integrated-band-body} with $q=n$ and $s=c_0k$. It gives a symmetric convex set $K\subseteq F$ with
	\begin{equation}\label{eqn:gamma_F(K)}
		\gamma_F(K)\geq \exp(-c_0k)
	\end{equation}
	and
	\begin{equation}\label{eqn:bounded_inner_product_K}
		\sup_{u\in\Sph^n}|\ip{P_FDA_u}{h}|
		\leq C_n\tau\norm{\xi}_{\mathrm{TV}}
		\max\left\{\sqrt{\Lambda},\sqrt{k}(c_0k)^{-1/(2n)}\right\},\quad \forall h\in K.
	\end{equation}
	By \eqref{eqn:low-crossing-lambda}, \eqref{eqn:tau}, and $k\asymp r$, we obtain that
	\begin{equation}\label{eqn:inner-product-bound}
		\sup_{u \in \Sph^n} |\ip{P_F D A_u}{h}| \leq C_n\frac{B}{m}r^{\frac12-\frac{1}{2n}}\norm{\xi}_{\mathrm{TV}},\qquad \forall h\in K.
	\end{equation}
	Theorem~\ref{thm:subspace-convex-set-pc} then yields
	$z \in (z^0 + K)\cap[-1,1]^k$ with at least $\kappa_0 k$ coordinates in
	$\{-1,+1\}$. Define $\theta = D(z-z^0)$ and $x' = x(\theta)$. Since
	$z - z^0 \in K \subseteq F$,
	\[
		\sum_{e \in \cE}\theta_e b_{\ell}(e) = \langle D b_{\ell},z - z^0\rangle = 0 \qquad \ell = 1,\dots,s.
	\]
	Now we look at the integrated discrepancy. If $n>1$ and $\norm{\xi}_{\mathrm{TV}}>0$, then by \eqref{eqn:inner-product-bound},
	\[
		\left|\sum_{e \in \cE}\theta_e A_u(e)\right|
		= |\langle P_F D A_u,z - z^0\rangle|
		\le C_n \frac{B}{m} r^{\frac{1}{2}-\frac{1}{2n}}\norm{\xi}_{\mathrm{TV}}.
	\]
	If $n=1$, then every band $\cB(u,t)$ crosses at most four edges of $\cE$ for
	each $t$, and therefore
	\[
		\left|\sum_{e \in \cE} \theta_e A_u(e)\right|
		\le \int_0^1 \sum_{e\in\cE}|\theta_e a_u(e,t)|\,d|\xi|(t)
		\le C_1\frac{B}{m}\norm{\xi}_{\mathrm{TV}}.
	\]
	The same bound is trivial when $\norm{\xi}_{\mathrm{TV}}=0$. Thus, in all cases,
	\begin{equation}\label{eqn:discrepancy-bound}
		\left|\sum_{e \in \cE}\theta_e A_u(e)\right|
		\le C_n \frac{B}{m} r^{\frac{1}{2}-\frac{1}{2n}}\norm{\xi}_{\mathrm{TV}}.
	\end{equation}

	\paragraph{Verification.}
	Now we verify the guarantees (i)--(iv). Since $\theta_e \in J_e$, $x' \in [-1,1]^m$,
	and unchanged coordinates stay unchanged. This shows (i). If $\theta_e$ is an endpoint of
	$J_e$, then at least one of $x_a'$ and $x_b'$ equals $\pm 1$, so at least
	$\kappa_0 k \ge \kappa r$ coordinates in $I$ are saturated, for some absolute constant $\kappa > 0$. This shows (ii).

	For exactness (iii), let $q \in V$. Since the restrictions of $q_1,\dots,q_s$ span
	$V|_I$, the vector $(q(y_i))_{i \in I}$ is a linear combination of the vectors
	$(q_{\ell}(y_i))_{i \in I}$. Therefore
	\[
		\sum_{i = 1}^m \alpha_i(x_i'-x_i)q(y_i) = \sum_{e = \{a,b\} \in \cE}\theta_e(q(y_a)-q(y_b)) = 0.
	\]
	Finally, for every $u\in \Sph^n$
	\[
		\int_0^1 \sum_{i = 1}^m \alpha_i(x_i'-x_i)\mathbf{1}_{\cB(u,t)}(y_i)\, d\xi(t) 
		= \int_0^1 \sum_{e \in \cE}\theta_e	a_u(e,t) \, d\xi(t) 
		= \sum_{e \in \cE}\theta_e A_u(e),
	\]
	and the estimate \eqref{eqn:discrepancy-bound} gives (iv).
\end{proof}

\subsection{Iteration of support reduction}
\label{sec:finite-band-sampling-proof}

\begin{proof}[Proof of Theorem~\ref{input:finite-band-sampling}]
	Without loss of generality, we may assume a balanced representation of
	$\mu$. If an atom of $\mu$ has mass $\rho \ge 2/M$, split it into
	$\lfloor \rho M/2\rfloor + 1$ equal pieces. The resulting representation of $\mu$
	has at most $3M/2$ atoms, each having mass at most $2/M$, and is therefore $3$-balanced.
	This splitting does not change the dimension of the restricted \(V\)-space. We use this balanced
	representation of $\mu$ below and merge coincident atoms at the end.

	Choose the constant $c_n$ small enough that
	\[
		c_n \le \min\left\{\frac{\bar c_n}{4},\frac{1}{4r_n}\right\}.
	\]

	We first describe the support-reduction step for an arbitrary current measure $\eta$.
	Let $\eta$ have a $B$-balanced representation
	\[
		\eta=\sum_i\alpha_i\delta_{y_i},
	\]
	let $Y=(y_1,\dots,y_m)$, and assume
	\[
		m>N\qquad\text{and}\qquad \dim V|_Y\leq c_nN.
	\]
	Since the constants belong to $V$, we have $1\leq\dim V|_Y\leq c_nN$,
	and hence $N/4\geq r_n$ by the choice of $c_n$.

	\paragraph{Support reduction step.}
	Starting from $x=0$, repeatedly apply Lemma~\ref{lem:weighted-fractional-band-colouring}
	to the currently unsaturated coordinates until at most $N/4$ coordinates
	remain unsaturated. At an intermediate stage, let $I$ be the index set of
	the unsaturated coordinates. Then $r:=|I|>N/4$ and
	\[
		\dim V|_I\leq\dim V|_Y\leq c_nN<4c_nr\leq\bar c_nr,
	\]
	and $r\geq r_n$, so Lemma~\ref{lem:weighted-fractional-band-colouring}
	is applicable.

	Each iteration saturates at least a constant fraction of the
	currently unsaturated coordinates. Thus the number of unsaturated coordinates decreases
	geometrically. Let $x\in[-1,1]^m$ be the final vector and set
	\[
		I^\ast=\{i:|x_i|<1\}.
	\]
	Then $|I^\ast|\leq N/4$ and
	\[
		\sum_{i=1}^m\alpha_ix_iq(y_i)=0\qquad\forall q\in V.
	\]
	Moreover, summing the band discrepancy increments over the geometrically
	decreasing unsaturated cardinalities gives
	\[
		\sup_{u\in\Sph^n} \left|\int_0^1\sum_{i=1}^{m} \alpha_i x_i \mathbf{1}_{\cB(u,t)}(y_i)\,d\xi(t)\right|
		\leq
		\begin{cases}
			C_1 B m^{-1}(1+\log(m/N)) \norm{\xi}_{\mathrm{TV}},&n=1,\\
			C_n B m^{-1/2-1/(2n)}\norm{\xi}_{\mathrm{TV}},&n\geq2.
		\end{cases}
	\]
	Indeed, when $n\geq 2$ and the unsaturated set has size $r$, the
	increment is bounded by
	\[
		C_n B m^{-1}r^{1/2-1/(2n)}\norm{\xi}_{\mathrm{TV}},
	\]
	and the resulting geometric sum is bounded by
	$C_n B m^{-1/2-1/(2n)}\norm{\xi}_{\mathrm{TV}}$. For $n=1$, there
	are $O(1+\log(m/N))$ iterations, each contributing at most
	$C_1 B m^{-1}\norm{\xi}_{\mathrm{TV}}$.

	For $\sigma\in\{-1,+1\}$, define
	\[
		\eta_{\sigma}=\sum_{i=1}^{m}
		\alpha_i(1+\sigma x_i)\delta_{y_i}.
	\]
	Since constants belong to $V$, we have $\sum_i\alpha_ix_i=0$, so each
	$\eta_{\sigma}$ has total mass one. Moreover, for every $q\in V$,
	\[
		\int_{\Sph^n}q(y)\,d\eta_{\sigma}(y)
		=\int_{\Sph^n}q(y)\,d\eta(y).
	\]
	Also,
	\[
		\eta_{\sigma}(\cB(u,t))-\eta(\cB(u,t))
		=\sigma\sum_{i=1}^{m}\alpha_ix_i
		\mathbf{1}_{\cB(u,t)}(y_i),
	\]
	so the preceding discrepancy estimate applies to either choice of
	$\sigma$.
	Choose $\eta_\sigma$ with a smaller number of atoms as the resulting
	measure. Since only coordinates in $I^\ast$ may remain fractional, the
	chosen $\eta_\sigma$ has at most
	\[
		\frac{m-|I^\ast|}{2}+|I^\ast|
		=\frac{m}{2}+\frac{|I^\ast|}{2}
		\leq\frac{m}{2}+\frac{N}{8}
	\]
	atoms, and every nonzero term has weight at most $2B/m$.

	\paragraph{Iteration.}
	Let the initial measure $\eta_0=\mu$, which has $m_0$ atoms (including
	multiplicity). If $m_0\leq N$, then merging coincident atoms in $\eta_0$
	already gives a valid choice of $\nu$. Thus we assume $m_0>N$. Let $Y_0$
	be the indexed family induced by the representation $\eta_0$, and set
	\[
		s=\dim V|_{Y_0}.
	\]
	Starting from $\eta_0$, repeatedly apply the support-reduction step,
	obtaining a sequence $\eta_1,\eta_2,\dots$, until $\eta_j$ has
	$m_j\leq N$ atoms. At every stage, the current indexed family $Y_j$ is a
	subfamily of $Y_0$, counted with multiplicity, and hence
	$\dim V|_{Y_j}\leq s\leq c_nN$. Thus the support-reduction step is
	applicable whenever $m_j>N$. The atom counts satisfy
	\[
		m_{j+1}\leq\frac{m_j}{2}+\frac{N}{8}
		<\frac58m_j.
	\]
	Thus the final representation contains at most $N$ atoms.

	We also need to keep the weights balanced. Initially $B_0=3$. Suppose at
	stage $j$, all weights are at most $B_j/m_j$. Since the next weights are
	at most $2B_j/m_j$, we may take
	\[
		B_{j+1}=2B_j\frac{m_{j+1}}{m_j}
		\leq B_j\left(1+\frac{N}{4m_j}\right).
	\]
	Since the sequence $m_j$ decreases geometrically until it reaches $N$,
	\[
		\prod_j\left(1+\frac{N}{4m_j}\right)
		\leq\exp\left(\sum_j\frac{N}{4m_j}\right)\leq C,
	\]
	where $C$ is an absolute constant. Thus all weight-balance constants
	$B_j$ are bounded by an absolute constant.

	Since $m_j$ decreases geometrically until it reaches $N$, we have
	\[
		\sum_j m_j^{-1/2-1/(2n)} \leq C_n N^{-1/2-1/(2n)}
	\]
	and
	\[
		\sum_j \frac{1+\log(m_j/N)}{m_j} = \frac{1}{N}\sum_j \frac{1+\log(m_j/N)}{m_j/N} \leq \frac{C}{N}.
	\]
	Thus, summing the band discrepancy increments gives
	\[
		\sup_{u \in \Sph^n}\left|\int_0^1\bigl(\eta^{\ast}(\cB(u,t)) - \mu(\cB(u,t))\bigr)\,d\xi(t)\right|
		\le C_n N^{-1/2-1/(2n)}\norm{\xi}_{\mathrm{TV}},
	\]
	where $\eta^{\ast}$ is the final measure. Exact agreement on $V$ is
	preserved at each stage, so
	\[
		\int_{\Sph^n}q(y)\,d\eta^{\ast}(y) = \int_{\Sph^n}q(y)\,d\mu(y) \qquad \forall
		q \in V.
	\]
	Merging coincident atoms in $\eta^{\ast}$ gives $\nu$, which is supported on
	at most $N$ points of $\supp(\mu)$ and retains both exactness on $V$ and the
	integrated band discrepancy bound.
\end{proof}

%!TeX root = main.tex

\section{Algorithmic remarks}\label{sec:algorithm}

In this section, we describe an algorithmic implementation of Theorem~\ref{thm:additive}. Its proof gives an algorithmic construction of $\nu$, which is summarized in Algorithm~\ref{alg:construction}. 

\begin{algorithm}
\caption{Algorithmic implementation of Theorem~\ref{thm:additive}}\label{alg:construction}
\begin{algorithmic}[1]
	\State Choose $N \asymp_{d,p} \eps^{-2(d-1)/(d+2p)}$ and $K\asymp_{d,p} \eps^{-2/(d+2p)}$
	\State Decompose $|t|^p = p_K(t) + b_K(t)$ and write $b_K(t)=\beta_K(|t|)$
	\While{$m > N$} 
		\State Replace the current representation $\mu$ with a $3$-balanced $\sum_{i=1}^m \alpha_i \delta_{y_i}$
		\State Initialize $x \in [-1,1]^m$ to $x=0$
		\While{more than $N/4$ coordinates of $x$ are unsaturated}
			\State $I = \{i: |x_i| < 1\}$
			\State Compute a low-crossing matching $\cE$ on the indexed family $((y_i)_{i\in I})$ \label{alg:line:low-crossing}
			\State Form $D$, $z^0$ and compute the equations that define $F$ in \eqref{eqn:F}
			\If {$d = 2$ or $\beta_K = 0$}
				\State set $\mathcal K = F$
			\Else 
				\State Apply Lemma~\ref{lem:integrated-band-body} to construct a convex set $\mathcal K\subseteq F$ satisfying \eqref{eqn:gamma_F(K)} and \eqref{eqn:bounded_inner_product_K} \label{alg:line:convex-set}
			\EndIf
			\State Apply the partial colouring theorem (Theorem~\ref{thm:subspace-convex-set-pc}) to obtain $z \in (z^0 + \mathcal K)\cap[-1,1]^k$ \label{alg:line:partial-colouring}
			\State Use $\theta = D(z-z^0)$ to update $x$
		\EndWhile
		\State Update $\mu$ based on $x$
	\EndWhile
\end{algorithmic}
\end{algorithm}

The main nontrivial steps are Lines \ref{alg:line:low-crossing}, \ref{alg:line:convex-set} and \ref{alg:line:partial-colouring}, as the other lines can be done deterministically in polynomial time. Note that the standard proof of the Jackson-type approximation in Lemma~\ref{lem:important_approx} is constructive. Combining the explicit constant in Jackson's approximation theorem~\cite{FKS09} with the estimates in Appendix~\ref{sec:L_1-approx}, all constants can be made explicit.

Regarding Line \ref{alg:line:low-crossing}, note that the relevant equatorial bands induce only
polynomially many distinct incidence patterns and can be enumerated in polynomial time; see~\cite[Theorem~3.3]{EORS86}. 
The low-crossing matching can be found by a deterministic polynomial-time algorithm; see the proof of \cite[Theorem~5.17]{matousek-discrepancy}.

Regarding Line \ref{alg:line:partial-colouring}, note that Rothvoss's partial colouring theorem is constructive
in randomized polynomial time~\cite{rothvoss}. 
Rothvoss's algorithm always returns a vector $z\in(z^0 + \mathcal K)\cap[-1,1]^k$, and with at least a constant probability, at least $\kappa_0k$ coordinates of $z$ are saturated. Thus, we repeat the algorithm until the required saturation progress occurs.
There are $T = O(\log(M/N))$ outer support-reduction
iterations, each containing at most $O(T)$ invocations of the
Rothvoss's partial-colouring algorithm. At every such invocation, perform
$O(\log T)$ independent trials and accept the first one that achieves
the required saturation progress. For sufficiently large constant in the big $O$-notation, the
failure probability of each invocation is at most $1/(C T^2)$ for a large absolute constant $C$.
A union bound over the $O(T^2)$ invocations gives overall success
probability at least $2/3$, while preserving polynomial running time.

The remaining question is Line \ref{alg:line:convex-set}. The proof of Lemma~\ref{lem:integrated-band-body} invokes Lemma~\ref{lem:chaining}, which uses a family of nets of $\cA = \{P_F D A_u: u \in \Sph^{d-1}\}\subseteq F$. The main difficulty is that the proof of the net size is existential, and so we need an algorithmic method to construct the convex set $\mathcal K$.

When $d=2$ or $\norm{\beta_K}_{\mathrm{TV}}=0$, the proof takes $\mathcal K=F$, so no net construction is needed. Hence, assume below that $d\geq 3$ and $\norm{\beta_K}_{\mathrm{TV}}>0$. Our idea is to construct a polynomial-size candidate set and prune it greedily; the packing lemma then guarantees the required size bound. For a pair $e = \{a,b\}$,
\[
	A_u(e)=\beta_K(|\ip{u}{y_b}|)-\beta_K(|\ip{u}{y_a}|),
\]
and hence
\begin{equation}\label{eqn:naive_bound}
	\norm{P_F D(A_u - A_v)}_2 \leq 2\tau \norm{\beta_K'}_\infty \sqrt{k}\norm{u - v}_2
\end{equation}
for $u,v\in\Sph^{d-1}$. By the construction in the proof of Lemma~\ref{lem:polynomial_approx}, $p_K'(\cos\theta)$ is a trigonometric polynomial of degree at most $K-1$ and
\[
	\int_0^{2\pi}|p_K'(\cos\theta)|\,d\theta
	\leq p\int_0^{2\pi}|G_{p-1}(\theta)|\,d\theta+C_pK^{-p}
	\leq C_p'.
\]
As mentioned above, the value of $C_p$ is tractable, and so is the value of $C_p'$.
Then, the standard $L_\infty$--$L_1$ inequality for trigonometric polynomials (see, e.g.~\cite[p.~229]{timan}) gives
\[
	\norm{p_K'}_\infty
	=\norm{p_K'(\cos\theta)}_\infty
	\leq \frac{K}{\pi}\int_0^{2\pi}|p_K'(\cos\theta)|\,d\theta
	\leq \frac{C_p'K}{\pi}.
\]
Since $\beta_K(t)=t^p-p_K(t)$ on $[0,1]$, it follows that
\begin{equation}\label{eqn:beta_K'_inf}
	\norm{\beta_K'}_\infty \leq p+\norm{p_K'}_\infty \leq p+\frac{C_p'K}{\pi}=:C_\beta.
\end{equation}
Therefore, if we want to construct a $\delta$-net of $\cA$, we can first construct an $\eta$-net $\mathcal{X}$ of $\Sph^{d-1}$, where
\[
	\eta = \frac{\delta}{8\tau C_\beta \sqrt{k}}.
\]
It follows from \eqref{eqn:naive_bound} and \eqref{eqn:beta_K'_inf} that
\[
	\mathcal{C} = \{P_F D A_x: x\in \mathcal{X}\}
\]
is a $(\delta/4)$-net of $\cA$, which is sufficient for the construction. This step can be done in polynomial time since $d$ is a constant and $|\mathcal{X}|$ is polynomial in $1/\eta$, hence polynomial in $m/\delta$. We can then greedily choose a maximal $(\delta/2)$-separated subset of $\mathcal{C}$, which is a $\delta$-net of $\cA$. The required net-size bound is the one supplied by the packing argument in Lemma~\ref{lem:integrated-band-body}.

The proof constructs $\mathcal K$ by retaining the first $J$ levels of the intersection in Lemma~\ref{lem:chaining}. Its defining slabs have the form
$\{h\in F: |\langle w,h\rangle|\leq A_{2d-2}\sqrt{j} \norm{w}_2\}$ for various vectors $w$. Here $A_{2d-2}$ denotes the constant $A_q$ chosen in Lemma~\ref{lem:chaining} with $q=2d-2$. By \eqref{eqn:inner-product-bound_chaining}, for all $v\in\cA$ and $h\in\mathcal K$,
\[
	|\ip{v}{h}| \leq 2^{-J/2}\widetilde{R}\norm{h}_2 + CA_{2d-2}\widetilde{R}.
\]
For the standard Gaussian vector $g_F$ in $F$, with probability at least $1-\exp(-c_1 k)$, it holds that $\norm{g_F}_2\leq C\sqrt{k}$ for some absolute constants $C, c_1 > 0$. Hence, we further intersect $\mathcal K$ with $B(0,C\sqrt{k})$, so every $h$ in the resulting set satisfies $\norm{h}_2\leq C\sqrt{k}$. Also, after adjusting the absolute constants, a union bound shows that the intersection still has Gaussian measure at least $\exp(-c_0 k)$.
It suffices to take $J = O(\log k)$ such that $2^{-J/2}C\sqrt{k}\leq C'A_{2d-2}$ for a suitable absolute constant $C'$. When choosing the chaining parameters, we may use the explicit upper bound $C_pK^{-p}$ in place of $\norm{\beta_K}_{\mathrm{TV}}$. It follows that, for every retained level $j\leq J$, the radius $2^{-j/2}\widetilde R$ is $1/\poly(m,K)$, so all the nets described above can be constructed in polynomial time. Thus, the final convex set $\mathcal K$ is the intersection of polynomially many slabs and one Euclidean ball. This description also gives a polynomial-time separation oracle for $\mathcal K$, which is needed for Rothvoss's partial colouring algorithm.

Overall, we conclude that the sparsified measure $\nu$ can be constructed by a randomized algorithm in $\poly(M,1/\eps)$ time with probability at least $2/3$, assuming the standard real-arithmetic model.

\section*{Acknowledgements}
The author is supported in part by the Singapore Ministry of Education AcRF Tier 1 grant RG21/25.

The author used ChatGPT 5.5 Plus during the development of this work to explore and refine proof strategies, formulate intermediate results, and generate initial drafts of parts of the technical proofs. All AI-generated material was substantially revised, corrected, and integrated by the author. All mathematical claims, proofs, and references were independently checked by the author.

\bibliographystyle{plain}
\bibliography{reference}

\appendix

%!TeX root = main.tex

\section{\texorpdfstring{Polynomial Approximation in $L_1$-Norm}{Polynomial Approximation in L1-Norm}}\label{sec:L_1-approx}

The goal of this section is to prove Lemma~\ref{lem:important_approx}. We first recall some basic definitions from approximation theory.

For a $2\pi$-periodic function $f$ that is integrable on $[0,2\pi]$, define its integral modulus of smoothness of order $r\geq 1$ by
\[
	\omega_r(f; \delta)_{L^1} = \sup_{|h|\leq \delta} \int_{0}^{2\pi} |\Delta_h^r f(\theta)|\,d\theta,
\]
where
\[
	\Delta_h^r f(\theta) = \sum_{j=0}^r (-1)^{r-j} \binom{r}{j} f(\theta + jh).
\]
If $f^{(r-1)}$ is absolutely continuous on the interval with endpoints $\theta$ and $\theta+rh$, we have the classical finite-difference identity (see, e.g.~\cite[p.~103]{timan})
\begin{equation}\label{eqn:modulus_integral_form}
	\Delta_h^r f(\theta) = \int_0^h\cdots\int_0^h f^{(r)}(\theta + t_1 + \cdots + t_r) \, dt_1\cdots dt_r.
\end{equation}
We also need the following classical result of Jackson-type approximation.

\begin{lemma}[{\cite[p.~325]{timan}}]
	Suppose that $r\geq 1$ is an integer, $f$ is $2\pi$-periodic and is $L^1$-integrable on $[0,2\pi]$. For any integer $K\geq 0$, there exists a trigonometric polynomial $p_K$ of degree at most $K$ such that
	\[
		\int_0^{2\pi} |f(t) - p_K(t)| \, dt \leq C_r \omega_r\left(f; \frac{1}{K+1}\right)_{L^1},
	\]
	where $C_r$ is a constant depending only on $r$.
\end{lemma}

The required estimate for the integral modulus is the following.
\begin{lemma}\label{lem:G_alpha}
	Let $\alpha \geq 0$ and let $r > \alpha+1$ be an integer. Define
	\[
		G_\alpha(\theta) :=
		\begin{cases}
			\sgn(\cos\theta)|\cos\theta|^\alpha, & \alpha>0,\\
			\sgn(\cos\theta), & \alpha=0.
		\end{cases}
	\]
	Then there exist constants $C_{\alpha,r}, \delta_r > 0$ such that
	\[
		\omega_r(G_\alpha;\delta)_{L^1} \leq C_{\alpha,r}\delta^{\alpha+1}
	\]
	for all $0<\delta<\delta_r$.
\end{lemma}
\begin{proof}
	We first consider the case $\alpha = 0$. It is easy to verify using the definition of $\omega_r(G_\alpha;\delta)_{L^1}$ that $\omega_r(G_0;\delta)_{L^1} \leq C_r \delta$.
	
	Now we assume that $\alpha > 0$.
	Choose $\delta_r \asymp 1/r$ such that $3r\delta_r<\pi/2$ and fix $\delta\in (0,\delta_r)$. Define $Z = \{\pi/2, 3\pi/2\}$.
	Fix $h$ with $|h|\leq \delta$, and let $I_h$ denote the interval with endpoints $0$ and $h$.
	We separate the integral over the neighbourhood $N(Z, 2r\delta)$ and its complement.
	
	Case 1: Suppose that $\theta\in N(Z, 2r\delta)$, then for $0\leq j\leq r$, $\theta + j h \in N(Z, 3r\delta)$ and so
	\[
		|G_\alpha(\theta + jh)| \leq (3r\delta)^\alpha.
	\]
	It follows that
	\[
		|\Delta_h^r G_\alpha(\theta)| \leq \sum_{j=0}^r \binom{r}{j} (3r\delta)^\alpha = 2^r (3r\delta)^\alpha
	\]
	and
	\[
		\int_{N(Z, 2r\delta)} |\Delta_h^r G_\alpha(\theta)| \, d\theta \leq 8r\delta\cdot 2^r (3r\delta)^\alpha \leq C_{\alpha,r} \delta^{\alpha+1}.
	\]
	
	Case 2: For every $\theta\notin N(Z,2r\delta)$, the interval with endpoints $\theta$ and $\theta+rh$ avoids $Z$. Hence $G_\alpha$ is smooth on this interval. Using \eqref{eqn:modulus_integral_form} and Tonelli's theorem, we have
	\[
		\int_{[0,2\pi]\setminus N(Z,2r\delta)} |\Delta_h^r G_\alpha(\theta)| \, d\theta \leq \int_{I_h} \cdots\int_{I_h}  	\int_{[0,2\pi]\setminus N(Z,2r\delta)} |G_\alpha^{(r)}(\theta + t_1 + \cdots + t_r)| \, d\theta dt_1\cdots dt_r
	\]
	For $t_1,\dots,t_r\in I_h$, we have $\theta + t_1+\cdots+t_r \not\in N(Z, r\delta)$ and so
	\[
		\int_{[0,2\pi]\setminus N(Z,2r\delta)} |\Delta_h^r G_\alpha(\theta)| \, d\theta \leq |h|^r \int_{[0,2\pi]\setminus N(Z,r\delta)} |G_\alpha^{(r)}(\theta)| \, d\theta.
	\]
	For $\theta\not\in N(Z, r\delta)$, write $\theta = z + x$ for some $z\in Z$ and $r\delta < |x| \le \pi/2$. Then
	\[
		G_\alpha(\theta) = \sigma_z \sgn(x)|x|^\alpha \left(\frac{\sin x}{x}\right)^\alpha,
	\]
	where $\sigma_{\pi/2}=-1$ and $\sigma_{3\pi/2} = 1$, and $(\sin x)/x$ is understood to be $1$ at $x=0$. It then follows that
	\[
		|G_\alpha^{(r)}(\theta)| \leq C_{\alpha,r} |x|^{\alpha-r}.
	\]
	Integrating, 
	\[
		\int_{[0,2\pi]\setminus N(Z,r\delta)} |G_\alpha^{(r)}(\theta)| \, d\theta \leq 4\int_{r\delta}^{\pi/2} C_{\alpha,r} x^{\alpha-r} dx \leq C_{\alpha,r}' \delta^{\alpha-r+1}.
	\]
	Therefore,
	\[
		\int_{[0,2\pi]\setminus N(Z,2r\delta)} |\Delta_h^r G_\alpha(\theta)| \, d\theta \leq C_{\alpha,r}' \delta^{\alpha+1}.
	\]
	
	The conclusion of the lemma follows by combining the two cases.
\end{proof}

Now we are ready to prove Lemma~\ref{lem:important_approx}.
\begin{proof}[Proof of {Lemma~\ref{lem:important_approx}}]
	Choose an integer $r > \alpha+1$, and let $G_\alpha$, $C_{\alpha,r}$ and $\delta_r$ be as in Lemma~\ref{lem:G_alpha}. Let $K_\alpha = \lceil 1/\delta_r\rceil$, then $1/(K+1) < \delta_r$ whenever $K \geq K_\alpha$.
	
	By the Jackson-type lemma and Lemma~\ref{lem:G_alpha}, there exists a trigonometric polynomial $T_K$ of degree at most $K$ such that 
	\[
		\int_0^{2\pi} |G_\alpha(\theta) - T_K(\theta)| \,d\theta
		\leq C_r \omega_r\left(G_\alpha; \frac{1}{K+1}\right)_{L^1}
		\leq C_{\alpha,r} (K+1)^{-\alpha-1}
		\leq C_\alpha K^{-\alpha-1}.
	\]
	Symmetrize $T_K$ by taking $T_K'(\theta) = (T_K(\theta) + T_K(-\theta))/2$, then $T_K'$ contains only cosine terms. Noticing that $G_\alpha(\theta)$ is an even function, we have
	\begin{align*}
		\int_0^{2\pi} |G_\alpha(\theta) - T_K'(\theta)| \,d\theta &\leq \frac{1}{2}\left(\int_0^{2\pi} |G_\alpha(\theta) - T_K(\theta)| \,d\theta + \int_0^{2\pi} |G_\alpha(\theta) - T_K(-\theta)| \,d\theta\right) \\
		&= \frac{1}{2}\left(\int_0^{2\pi} |G_\alpha(\theta) - T_K(\theta)| \,d\theta + \int_0^{2\pi} |G_\alpha(-\theta) - T_K(-\theta)| \,d\theta\right) \\
		&= \int_0^{2\pi} |G_\alpha(\theta) - T_K(\theta)| \,d\theta \\
		&\leq C_\alpha K^{-\alpha-1}.
	\end{align*}
	Since $\cos(n\theta)$ can be written as a polynomial in $\cos\theta$, we can write $T_K'(\theta) = q_K(\cos\theta)$ for some polynomial $q_K$ of degree at most $K$. Define
	\[
		p_K(t) = \frac{q_K(t) - q_K(-t)}{2}
	\]
	then $p_K$ is an odd polynomial of degree at most $K$.
	
	It follows that
	\begin{align*}
		\int_{-1}^1 \bigl| \sgn(t)|t|^\alpha - p_K(t) \bigr| \, dt 
		&= 
		\int_{-1}^1 \left| \sgn(t)|t|^\alpha - \frac{q_K(t) - q_K(-t)}{2} \right| \, dt \\
		&\leq
		\frac{1}{2}\int_{-1}^1 \bigl| \sgn(t)|t|^\alpha - q_K(t) \bigr| \, dt
		+
		\frac{1}{2}\int_{-1}^1 \bigl| \sgn(t)|t|^\alpha + q_K(-t) \bigr| \, dt \\
		&=
		\int_{-1}^1 \bigl| \sgn(t)|t|^\alpha - q_K(t) \bigr| \, dt \\
		&\leq
		\frac{1}{2}\int_{0}^{2\pi} \bigl| G_\alpha(\theta) - T_K'(\theta) \bigr| \, d\theta \\
		&\leq C_\alpha K^{-\alpha-1}. \qedhere
	\end{align*}
\end{proof}

%!TeX root = main.tex
\section{Packing Lemma}\label{sec:packing}

In this section, we prove Lemma~\ref{lem:packing}, the probability space version of the packing lemma. 
We shall discretize the probability space $\Omega$ by sampling and use the following discrete packing lemma from \cite[Lemma 5.14]{matousek-discrepancy}.
\begin{lemma}\label{lem:packing-matousek}
	Let $(X,\cR)$ be a set system on a $n$-point set. Define the metric on $\cR$ as $\rho(R_1,R_2) = |R_1\triangle R_2|$. Suppose that there exist constants $C_0 > 0$ and $q>1$ such that 
	\[
		\pi_{\cR}(m) \leq C_0 m^q, \qquad \forall m\geq 1.
	\]
	Then there exists a constant $C_{q, C_0} > 0$ such that for every integer $\delta$ with $1\leq \delta\leq n$, every $\delta$-separated family $\cP\subseteq \cR$ satisfies 
	\[
		|\cP| \leq C_{q, C_0} (n/\delta)^{q}.
	\]
\end{lemma}

Now we are ready to state the proof of Lemma~\ref{lem:packing}.

\begin{proof}[Proof of Lemma~\ref{lem:packing}]
	Without loss of generality, assume that $\cP$ is finite and $|\cP|\geq 2$. Let
	$M=\lceil C|\cP|/\delta\rceil$, where $C > 0$ is a sufficiently large absolute
	constant. Choose independent samples $\omega_1,\dots,\omega_M\in \Omega$ and
	define, for each $R\in\cR$, the sampled set
	$\hat{R} = \{j\in [M]: \omega_j\in R\}$. This induces a finite set family
	$([M], \hat{\cR})$ satisfying
	\[
		\pi_{\hat{\cR}}(m)\leq \pi_{\cR}(m) \leq C_0 m^q,\quad \forall m\geq 1.
	\]
	
	Suppose that $R_1,R_2\in \cR$. Then
	\[
		\mu(R_1\triangle R_2) = \frac{\E|\hat{R}_1 \triangle \hat{R}_2|}{M}.
	\]
	For distinct $R_1,R_2\in\cP$, by a Chernoff bound,
	\[
		\Pr_{\omega_1,\dots,\omega_M}\left\{ |\hat{R}_1 \triangle \hat{R}_2| \leq \frac{\delta M}{2} \right\} < e^{-c\delta M}
	\]
	for some absolute constant $c>0$. Taking a union bound over at most
	$|\cP|^2$ pairs $(R_1,R_2)$ and choosing $C$ sufficiently large, we obtain
	that
	\[
		 |\hat{R}_1 \triangle \hat{R}_2| > \frac{\delta M}{2}
	\]
	with positive probability for all distinct $\hat{R}_1, \hat{R}_2\in
	\hat{\cP}$. Thus there exist $\omega_1,\dots,\omega_M$ such that
	$|\hat{\cP}| = |\cP|$. Setting $\delta' = \lfloor\delta M/2\rfloor$, we have
	$1\leq\delta'\leq M$ and $\hat{\cP}$ is $\delta'$-separated. By
	Lemma~\ref{lem:packing-matousek},
	\[
		|\cP|=|\hat{\cP}|\leq C_{q, C_0}
		\left(\frac{M}{\delta'}\right)^{q}
		\leq C'_{q,C_0}\delta^{-q}. \qedhere
	\]
\end{proof}

%!TeX root = main.tex
\section{\texorpdfstring{Embedding $\ell_2^d$}{Embedding l2 in dimension d}}\label{sec:ball}

Let $\sigma_{d-1}$ be the uniform probability measure on $\Sph^{d-1}$ and
\[
	m_{d,p} = \int_{\Sph^{d-1}} |\ip{u}{y}|^p d \sigma_{d-1}(y),
\]
which is a constant independent of $u\in\Sph^{d-1}$.

We shall prove the following theorem for embedding $\ell_2^d$ by partitioning spheres and concentration inequalities, following the approach developed in \cite{BL88}. The main additional ingredients are cubature exactness for higher-degree polynomials, and polynomial approximation of $|t|^p$ on cells away from the equator, where the approximation error decays exponentially with the degree.

\begin{theorem}\label{thm:ball}
	Suppose that $d\geq 2$ and $p\geq 1$. There exist constants $C_{d,p},B_{d,p}>0$ such that, for every $\eps\in (0,1/2)$, there are points $z_1,\dots,z_N\in \Sph^{d-1}$ with positive weights $w_1,\dots,w_N$ such that $\sum_i w_i = 1$ and
	\[
		\sup_{u\in \Sph^{d-1}} \left| \sum_{i=1}^N w_i |\ip{u}{z_i}|^p - m_{d,p} \right|\leq \eps
	\]
	with
	\[
		N\leq C_{d,p}\left(\frac{1}{\eps}\right)^{\frac{2(d-1)}{d+2p}}\left(\log\frac{1}{\eps}\right)^{B_{d,p}}.
	\]
\end{theorem}

\subsection{An auxiliary cubature lemma}

\begin{lemma}\label{lem:cubature}
	 Let $X\subseteq \Sph^n$ be compact, let $\mu$ be a finite positive Borel measure supported on $X$, and let $V\subseteq C(X)$ be a finite-dimensional subspace containing the constants.
	 
	 There exists a random positive atomic measure $\nu$ which satisfies the following properties:
	 \begin{enumerate}[label=(\roman*),itemsep=0pt]
	 	\item $|\supp(\nu)| \leq \dim V$;
	 	\item $\nu$ and $\mu$ are exact on $V$ almost surely, i.e.
	 	\[
	 		\int_X q(y) d\nu(y) = \int_X q(y) d\mu(y),\qquad \forall q\in V
	 	\]
	 	almost surely;
	 	\item for every continuous function $f\in C(X)$,
	 	\[
	 		\E_\nu \int_X f(y)d\nu(y) = \int_X f(y)d\mu(y).
	 	\]
	 \end{enumerate}
\end{lemma}
\begin{proof}
	If $\mu(X)=0$, take $\nu=0$. Otherwise, by scaling, it suffices to prove the result when $\mu(X) = 1$. Suppose that $s=\dim V$ and choose a basis $1,q_2,\dots,q_s$ of $V$. Let $\Sigma$ be the set of all probability measures $\sigma$ on $X$ satisfying
	\[
		\int_X q_j(y) d\sigma(y) = \int_X q_j(y) d\mu(y), \quad \forall j=2,\dots,s.
	\]
	It is clear that $\Sigma\neq\emptyset$ since $\mu\in \Sigma$. One can also verify that $\Sigma$ is convex and that it is compact in weak-$^*$ topology.
	
	We claim that every extreme point of $\Sigma$ is supported on at most $s$ points. Suppose otherwise. Then there exists an extreme point $\sigma\in\Sigma$ with $|\supp(\sigma)| > s$. We can find $s+1$ mutually disjoint Borel subsets $A_1,\dots,A_{s+1}\subset X$ (by taking small neighbourhoods of points in $\supp(\sigma)$) with $\sigma(A_i) > 0$ for all $i=1,\dots,s+1$. Define
	\[
		u_i = \left(\int_{A_i} 1 d\sigma(y), \int_{A_i} q_2(y) d\sigma(y), \dots, \int_{A_i} q_s(y) d\sigma(y) \right)\in \R^{s},
	\]
	then $u_1,\dots,u_{s+1}$ are $s+1$ vectors in $\R^s$ and must be linearly dependent, that is, there exist $\alpha_1,\dots,\alpha_{s+1}$ such that $\sum_i \alpha_i u_i = 0$. Let
	\[
		\tau = \sum_{i=1}^{s+1} \alpha_i \sigma|_{A_i}.
	\]
	Then $\int_{X} q_i(y) d\tau(y) = 0$ for all $i$, which implies that
	\[
		\int_{X} q(y) d\tau(y) = 0,\qquad \forall q\in V.
	\]
	For sufficiently small $\eps > 0$, both $\sigma+\eps\tau$ and $\sigma-\eps\tau$ are probability measures in $\Sigma$, contradicting the extremality of $\sigma$. This proves our claim.
	
	Now, by Choquet's Theorem, for $\mu\in\Sigma$, there exists a probability distribution $\cD$ on the extreme points of $\Sigma$ such that its expectation is $\mu$ in the weak-$^\ast$ sense. Let $\nu$ be drawn from $\mathcal{D}$. Then, $|\supp(\nu)|\leq s$ by our claim above. This proves property (i). Property (ii) follows from the fact that $\nu\in\Sigma$. Finally, since $\mu$ is the mean of $\cD$ in the weak-$^\ast$ sense, it holds automatically that 
	\[
		\int_X f(y)d\mu(y) = \E_{\nu\sim \cD} \int_X f(y)d\nu(y),\qquad f\in C(X).
	\]
	The proof of the lemma is complete.
\end{proof}

\subsection{Proof of Theorem~\ref{thm:ball}}

We shall also need the following result on polynomial approximation of analytic functions.

\begin{lemma}[{\cite[p.~280]{timan}}]\label{lem:approximation_analytic}
	Suppose that $a<b$, $\rho > (b-a)/2$ and $f$ is analytic in the ellipse $E_{\rho}(a,b)$ with foci $a,b$ and sum of semiaxes equal to $\rho$. Then for every
	\[
		q\in \left(\frac{b-a}{2\rho}, 1\right),
	\]
	there exists $C_{f,q}>0$ such that for every $K\geq 1$, there is a polynomial $p_K$ of degree at most $K$ satisfying
	\[
		\norm{f-p_K}_{L^\infty[a,b]}\leq C_{f,q}q^K.
	\]
\end{lemma}

Now we present the proof of Theorem~\ref{thm:ball}.
\begin{proof}[Proof of Theorem~\ref{thm:ball}]
	As before, we set $n = d-1$.
	
	Let $N$ be an integer to be fixed later, and set $\eta \asymp N^{-1/n}$. We write $\Sph^n=Q_1\cup\cdots\cup Q_N$, where the $Q_j$ are compact connected sets satisfying $\sigma_n(Q_j) = 1/N$ and $\diam(Q_j)\leq c_n \eta$ for every $j=1,\dots,N$, and distinct cells may intersect on sets of zero $\sigma_n$-measure.
	
	Let $K \asymp \log N$. For each $Q_j$, independently apply Lemma~\ref{lem:cubature} to $\sigma_n|_{Q_j}$ and $V = \cP_K(\Sph^n)$. This yields a random measure $\nu_j$ on $Q_j$ such that $\nu_j(Q_j) = \sigma_n(Q_j)$, and
	\[
		\int_{Q_j} q(y)d\nu_j(y) = \int_{Q_j} q(y)d\sigma_{n}(y),\quad \forall q\in \cP_K(\Sph^n)
	\]
	almost surely and
	\[
		\E \int_{Q_j} f(y) d\nu_j(y) = \int_{Q_j} f(y)d\sigma_{n}(y),\quad \forall f\in C(Q_j).
	\]
	Define
	\[
		\nu = \sum_{j=1}^N \nu_j,
	\]
	then $\nu(\Sph^n) = 1$, so $\nu$ is a probability measure on $\Sph^n$, and
	\[
		|\supp(\nu)| \leq N \dim \cP_K(\Sph^n) \leq C_n N K^n = O_{n,p}(N \log^n N).
	\]
	
	Now we investigate the approximation error. Fix $u\in \Sph^n$. First, consider the cells $Q_j$ away from the equator $u^\perp$, that is,
	\[
		|\ip{u}{y}| > \eta,\qquad \forall y\in Q_j.
	\]
	Then the interval $I_j=\{\ip{u}{y}: y\in Q_j\}$
	has length at most $C_n\eta$ and is at distance at least $\eta$ from $0$. If $I_j$ is a singleton, take $p_j$ to be the constant value of $|\cdot|^p$ on $I_j$. 
	Otherwise, under a linear transformation, the problem reduces to approximating $s\mapsto s^p$ on the fixed interval $[1,1+C_n]$. This function has an analytic extension to the ellipse $E_{\rho_n}(1,1+C_n)$ for some $\rho_n > C_n/2$ depending only on $n$. Lemma~\ref{lem:approximation_analytic}, followed by the inverse transformation, gives a polynomial $p_j$ of degree at most $K$ satisfying
	\[
		\norm{ |\ip{u}{\cdot}|^p - p_j(\ip{u}{\cdot}) }_{L^\infty(Q_j)} \leq C_{n,p} e^{-c_{n,p}K},
	\]
	where $C_{n,p}$ and $c_{n,p}$ are constants independent of $u$ and $Q_j$. By the exactness of $\nu_j$, 
	\[
		\int_{Q_j} p_j(\ip{u}{y}) d\nu_j(y) = \int_{Q_j} p_j(\ip{u}{y}) d\sigma_n(y),
	\]
	so on each cell $Q_j$,
	\[
		\left| \int_{Q_j} |\ip{u}{y}|^p d\nu_j(y) - \int_{Q_j} |\ip{u}{y}|^p d\sigma_n(y) \right| \leq \frac{C_{n,p}}{N}e^{-c_{n,p}K}.
	\]
	Therefore, summing the cellwise errors over all those cells $Q_j$, the approximation error is bounded by
	\begin{equation}\label{eqn:far_error}
		C_{n,p} e^{-c_{n,p}K} \leq C_0\eta^{\frac{n+1}{2}+p},
	\end{equation}
	by choosing the constant in $K\asymp \log N$ sufficiently large and the constant in $\eta\asymp N^{-1/n}$ sufficiently small.
	
	Next we consider the cells around the equator. On these cells $Q_j$ we have $|\langle u,y\rangle|\leq C_n\eta$ for all $y\in Q_j$, where $C_n>0$ depends only on $n$. Define a random variable
	\[
		X_{u,j} = \int_{Q_j} |\ip{u}{y}|^p d\nu_j(y) - \int_{Q_j} |\ip{u}{y}|^p d\sigma_n(y),
	\]
	then $\E X_{u,j} = 0$,
	\[
		|X_{u,j}|\leq 2C_n^p\eta^p \cdot \sigma_n(Q_j) = \frac{2C_n^p\eta^p}{N} =: L,\qquad  \Var(X_{u,j}) = \E X_{u,j}^2 \leq L^2.
	\]
	There are at most $J:=C'_n\eta^{-(n-1)}$ near-equator cells.
	Summing up the error $X_{u,j}$ over all these cells, by Hoeffding's inequality,
	\[
		\Pr\left\{ \left| \sum_j X_{u,j} \right| > t \sqrt{J}L \right\}
		\leq 2e^{-t^2/8}.
	\]
	
	Let $\cN$ be an $\zeta$-net in $\Sph^n$ with $|\cN|\leq (C_n/\zeta)^n$, where $\zeta>0$ is a parameter to be determined. Taking $t\asymp_n \sqrt{\log(1/\zeta)}$, we have
	\[
		\left(\frac{C_n}{\zeta}\right)^n \cdot 2e^{-\frac{t^2}{8}} < 1,
	\]
	so there exists $\nu$ such that 
	\begin{equation}\label{eqn:near_error}
		 \left| \sum_j X_{u,j} \right|\leq t\sqrt{J}L \leq C_1\eta^{\frac{n+1}{2}+p}\sqrt{\log(1/\zeta)}
	\end{equation}
	for all $u\in\cN$, where $C_1$ depends only on $n,p$.
	
	Choose $\eta$ and $\zeta$ such that 
	\begin{equation}\label{eqn:eta_constraint}
		(C_0+C_1)\eta^{\frac{n+1}{2}+p}\sqrt{\log(1/\zeta)} < \frac{\eps}{2}.
	\end{equation}
	Then it follows from \eqref{eqn:far_error} and \eqref{eqn:near_error} that for all $u\in \cN$,
	\[
		\left| \int_{\Sph^{n}} |\ip{u}{y}|^p d\nu(y) - \int_{\Sph^{n}} |\ip{u}{y}|^p d\sigma_n(y) \right| < \frac{\eps}{2}.
	\]
	For a general $u\in\Sph^n$, choose $u'\in \cN$ with $\norm{u-u'}_2\leq \zeta$. Since
	\[
		\left| |\ip{u}{y}|^p - |\ip{u'}{y}|^p \right| \leq p\norm{u-u'}_2,
	\]
	we have
	\begin{align*}
		\left| \int_{\Sph^{n}} |\ip{u}{y}|^p d\nu(y) - \int_{\Sph^{n}} |\ip{u}{y}|^p d\sigma_n(y) \right| &\leq 
		\left| \int_{\Sph^{n}} |\ip{u'}{y}|^p d\nu(y) - \int_{\Sph^{n}} |\ip{u'}{y}|^p d\sigma_n(y) \right| + 2p\zeta \\
		&< \frac{\eps}{2} + \frac{\eps}{2} = \eps,
	\end{align*}
	provided that
	\[
		\zeta = \frac{\eps}{4p}.
	\]
	Plugging in the value of $\zeta$ into \eqref{eqn:eta_constraint}, we see $\eta$ can be chosen as
	\[
		\eta \asymp_{n,p} \left(\frac{\eps}{\sqrt{\log(1/\eps)}}\right)^{2/(2p+n+1)}.
	\]	
	Consequently,
	\begin{align*}
		|\supp(\nu)| &\lesssim_{n,p} N \log^n N \\
		&\asymp_n \left(\frac{1}{\eta}\log \frac{1}{\eta} \right)^n \\
		&\lesssim_{n,p} \left(\frac{1}{\eps}\right)^{\frac{2n}{n+2p+1}}\left(\log\frac{1}{\eps}\right)^{B_{d,p}}.
	\end{align*}

\end{proof}

%!TeX root = main.tex
\section{Sparsification based on sphere partitioning}\label{sec:general_bourgain}

We prove the following theorem in this section.
\begin{theorem}\label{thm:general_bourgain}
	Suppose that $d\geq 2$ and $p\geq 1$. There exist constants $C_{d,p},B_{d,p}>0$ such that, for every $\eps\in (0,1/2)$ and every probability measure $\mu$ on $\Sph^{d-1}$, there is a positive atomic probability measure
	\[
		\nu = \sum_{i=1}^N w_i \delta_{\xi_i},\qquad \xi_i\in\Sph^{d-1},\quad w_i > 0, 
	\]
	such that
	\[
		\sup_{u\in \Sph^{d-1}} \left| \int_{\Sph^{d-1}} |\ip{u}{y}|^p d\nu(y) - \int_{\Sph^{d-1}} |\ip{u}{y}|^p d\mu(y)  \right|\leq \eps
	\]
	with
	\[
		N\leq C_{d,p}\left(\frac{1}{\eps}\right)^{\frac{2(d-1)}{d+2p-1}}\left(\log\frac{1}{\eps}\right)^{B_{d,p}}.
	\]
\end{theorem}

\begin{proof}
	The proof is similar to that of Theorem~\ref{thm:ball}, so we describe only the differences.
	
	Again, let $n=d-1$, $\eta\asymp N^{-1/n}$, and $K \asymp_{n,p}\log N$. We first partition $\Sph^n = Q_1\cup\cdots\cup Q_N$ into compact connected sets of equal area with $\diam(Q_i)\leq C_n\eta$, where distinct cells may intersect on sets of $\sigma_n$-measure zero. Define the disjoint Borel sets
	\[
		A_1=Q_1,\qquad A_j=Q_j\setminus\bigcup_{i<j}Q_i,\quad j=2,\dots,N,
	\]
	so that $\Sph^n=A_1\cup\cdots\cup A_N$ is a partition and $A_j\subseteq Q_j$. Let $\mu_j = \mu|_{A_j}$, so that $\mu = \sum_{j=1}^N \mu_j$. Although $A_j$ need not be closed, $\supp(\mu_j)$ is a closed subset of the compact set $Q_j$ and is therefore compact.
	
	We decompose each $\mu_j$ as a sum of at most $\lceil N \mu_j(\Sph^n)\rceil$ nonzero positive submeasures, each of total mass at most $1/N$. In this decomposition, atoms of $\mu$ may be split into several submeasures supported at the same point. The total number of resulting submeasures is at most
	\[
		\sum_{j=1}^N\left\lceil N\mu_j(\Sph^n)\right\rceil
		\leq N\left(\sum_{j=1}^N\mu_j(\Sph^n) + 1\right) = 2N.
	\]
	From this point onward, we relabel all these submeasures as $\mu_1,\dots,\mu_{N'}$, with $N'\leq 2N$. Each relabeled $\mu_j$ is associated with a parent cell $Q_i$, has compact support contained in $Q_i$, and satisfies $\mu_j(\Sph^n)\leq 1/N$.
	
	For each $j$, independently apply Lemma~\ref{lem:cubature} to $\mu_j$ on the compact space $\supp(\mu_j)$, with $V=\cP_K(\Sph^n)$ restricted to this space, obtaining a random positive atomic measure $\nu_j$ supported on $\supp(\mu_j)$. Let $\nu=\sum_{j=1}^{N'}\nu_j$. Exactness on the constants gives $\nu(\Sph^n)=\sum_j\mu_j(\Sph^n)=1$, so $\nu$ is a probability measure. The approximation error on the away-from-equator submeasures is then again bounded by \eqref{eqn:far_error}.
	
	Fix $u\in\Sph^n$. Call a submeasure $\mu_j$ near the equator if its parent cell is not away from $u^\perp$. For every such submeasure, $|\ip{u}{y}|\leq C_n\eta$ on $\supp(\mu_j)$. Define
	\[
		X_{u,j}=\int_{\Sph^n}|\ip{u}{y}|^p d\nu_j(y)
		-\int_{\Sph^n}|\ip{u}{y}|^p d\mu_j(y).
	\]
	Then the $X_{u,j}$ are independent, have mean zero, and satisfy
	\[
		|X_{u,j}| \leq 2C_n^p \eta^p \cdot \mu_j(\Sph^n) \leq \frac{2C_n^p \eta^p}{N} =: L.
	\]
	There are at most $J\leq N'\leq 2N\asymp \eta^{-n}$ near-equator submeasures. Thus \eqref{eqn:near_error} now becomes
	\[
		\left| \sum_j X_{u,j} \right| \leq C_1 \eta^{\frac{n}{2}+p} \sqrt{\log(1/\zeta)}.
	\]
	The rest of the proof remains nearly identical, except that we now choose $\eta$ such that
	\[
		\eta^{\frac{n}{2}+p} \sqrt{\log(1/\eps)} \asymp_{n,p} \eps,
	\]
	which yields
	\[
		\eta \asymp_{n,p} \left( \frac{\eps}{\sqrt{\log(1/\eps)}} \right)^{\frac{2}{n+2p}}
	\]
	and thus
	\[
		|\supp(\nu)|\leq N'\dim\cP_K(\Sph^n)
		\lesssim_{n,p} \left(\frac{1}{\eta}\log \frac{1}{\eta} \right)^n 
		\lesssim_{n,p} \left(\frac{1}{\eps}\right)^{\frac{2n}{n+2p}}\left(\log\frac{1}{\eps}\right)^{B_{d,p}}.
	\]
\end{proof}

\end{document}